\documentclass[12pt,a4paper]{article}

\usepackage[latin1]{inputenc}
\usepackage[english]{babel}
\usepackage{amsmath}
\usepackage{color}
\usepackage{amsfonts,enumitem}
\usepackage{amssymb}
\usepackage{hyperref} 
\usepackage{graphicx,tikz}

\usepackage{algorithmic}
\usepackage[ruled,noend]{algorithm2e}
\usepackage[skins]{tcolorbox}

\newcommand{\crit}{\mathrm{crit}}
\newcommand{\conv}{\mathrm{conv}}
\newcommand{\cl}{\mathrm{cl}\,}
\newcommand{\parents}{\mathrm{\texttt{parents}}}

\newcommand{\J}{\mathcal{J}}

\newcommand{\XX}{\mathcal{X}}
\newcommand{\YY}{\mathcal{Y}}
\newcommand{\SSS}{\mathcal{S}}
\newcommand{\HH}{\mathcal{H}}
\newcommand{\BB}{\mathcal{B}}

\newcommand{\EE}{\mathbb{E}}
\newcommand{\RR}{\mathbb{R}}
\newcommand{\NN}{\mathbb{N}}
\newcommand{\dd}{{\rm d}}

\newcommand{\partialc}{\partial^c}

\newtheorem{theorem}{Theorem}
\newtheorem{lemma}{Lemma}
\newtheorem{proposition}{Proposition}
\newtheorem{corollary}{Corollary}

\newtheorem{definition}{Definition}

\newtheorem{remark}{Remark}

\newtheorem{example}{Example}
\newenvironment{proof}[1][]{\noindent {\bf Proof #1:\;}}{\hfill $\Box$}

\providecommand{\keywords}[1]{\textbf{\textbf{Keywords. }} #1}

\newcommand{\grad}{\mathrm{grad}\,}
\newcommand{\rank}{\mathrm{rank}}
\newcommand{\R}{\mathbb{R}}

\title{Conservative set valued fields, automatic differentiation, stochastic gradient methods  and deep learning}
\begin{document}
\author{
   J\'{e}r\^{o}me Bolte\footnote{Toulouse School of Economics, Universit\'{e} Toulouse 1 Capitole, France.}   \, and Edouard Pauwels\footnote{IRIT, Universit\'e de Toulouse, CNRS. DEEL, IRT Saint Exupery. Toulouse, France.}
}

\date{\today}

\maketitle

\begin{abstract}

Modern problems in AI or in numerical analysis require nonsmooth approaches with a flexible calculus. We introduce generalized derivatives called conservative fields for which we develop a calculus and provide representation formulas. Functions having a conservative field are called path differentiable:  convex, concave,  Clarke regular and any semialgebraic Lipschitz continuous functions are path differentiable. Using Whitney stratification techniques for semialgebraic and definable sets, our model provides  variational formulas for nonsmooth automatic differentiation oracles, as for instance the famous backpropagation algorithm in deep learning. Our differential model is applied to establish the convergence in values of nonsmooth stochastic gradient methods as they are implemented in practice.

\end{abstract}


\newpage
\tableofcontents
\newpage
\section{Introduction}
%\paragraph{Conservative fields:}
Classical approaches to solution methods for nonsmooth equations or nonsmooth  optimization come from the calculus of variations  \cite{moreau1963fonctionnelles,rockafellar1963convex,aubin2009set,borwein2010convex,clarke1983optimization,ioffe2017,mordukhovitch2006variational,rockafellar1998Variational}. They have been successfully used in several contexts, from partial differential equations  to machine learning. But  most of the advances made in these last decades  apply to classes revolving around convex-like non differentability phenomena: convex functions, semiconvex functions or (Clarke) regular problems. On the other hand several major problems arising in machine learning, numerical analysis, or non regular dynamical systems  are not covered by these regularity models due to various calculus restrictions and the necessity of decom\-po\-sing algorithms, see e.g., \cite{bottou2018optimization} and references therein.  We propose a notion of gene\-ra\-li\-zed derivatives  and identify a class of locally Lipschitz functions, called path
 differentiable functions, for which we obtain a flexible calculus, should we accept to use weaker gene\-ra\-li\-zed derivatives than standard ones. Our starting point is extremely elementary, we see derivation as an inverse operation to integration:
\begin{equation*}%\label{intdiff}
    f(y)-f(x)=\int_x^y f'(t)\dd t.
\end{equation*}
We thus introduce and study graph closed set valued mappings $D_f:\R^p\rightrightarrows \R^p$, and locally Lipschitz functions $f \colon \RR^p \to \RR$ related by
\begin{equation*}\label{intdiffNonSmooth} 
f(\gamma(1))-f(\gamma(0)) = \int_{0}^1 \langle D_f(\gamma(t)),\dot \gamma(t)\rangle \dd t, \qquad \forall \gamma \in AC([0,1],\RR^p),
\end{equation*}
where we use Aumann's integration  while $AC([0,1], \RR^p)$ is the set of absolutely continuous functions from $[0,1]$ to $\RR^p$. Rephrasing this property yields a generalized form of the zero circulation property 
\begin{equation*}%\label{circ}
\int_{0}^1\left\langle D(\gamma(t)), \dot \gamma(t)\right\rangle \dd t=\left\{0\right\}, \qquad \forall \gamma \in AC([0,1],\RR^p), \:\gamma(0)=\gamma(1).
\end{equation*}
We naturally call these objects $D$ conservative set valued fields and a function having a conservative field is called path differentiable. Convex, concave, Clarke regular, but also  {\em any} semialgebraic Lipschitz continuous functions or Whitney stratifiable functions are path differentiable. 

We provide a calculus and several  characterizations of conservativity. First we show that conservative fields are classical gradients almost everywhere, which makes  the Clarke sub-dif\-fe\-ren\-tial a minimal convex conservative field.
Second, in the framework of semialgebraic or o-minimal structures, we provide conservative fields with a variational stratification formula \cite{bolte2007clarke}. This connection between Whitney stratification and conservativity allows to generalize known qualitative properties from the smooth world to definable conservative fields: Morse-Sard theorem, nonsmooth Kurdyka-\L ojasiewicz inequality  and convergence of solutions of differential inclusions ruled by a conservative field.
 
On a more applied side, conservative fields allow to analyze fundamental modern numerical algorithms in machine learning or numerical analysis based on automatic differentiation \cite{rumelhart1986learning,griewank2008evaluating} and decomposition \cite{bottou2018optimization,davis2018stochastic} in a nonsmooth context. Automatic differentiation is indeed proved to yield conservative fields which allows in turn to study discrete stochastic algorithms that are massively used to train AI systems. 
We illustrate this with the problem of training nonsmooth deep neural networks which are designed to perform prediction tasks based on a large labeled database  \cite{lecun2015deep}. 

Our work connects very applied concerns  with the recent theory of o-minimal structures, by revealing surprising links  between  the massively used numerical libraries (Tensorflow \cite{abadi2016tensorflow}, Pytorch, \cite{paszke2017workshop}), and   Whitney stratifications.

\paragraph{Structure of the paper}  As a conclusion, let us mention that this article contains material that can be considered as having distinct and independent interests. Researchers working in  analysis may focus on Sections 2 and 3, which present conservative fields. Those having affinity with geometry can also go through  Section 4 which provides insights into the semialgebraic and the definable cases. The more applied sections provide theorems which we believe useful to several communities: Section 5 is on nonsmooth automatic differentiation and a theoretical model for the corresponding ``oracle", while Section 6  studies stochastic gradient descent (with mini-batches) and deep learning.

\section{Conservative set valued fields}

\paragraph{Notations.}
We restrict our analysis to locally Lipschitz continuous functions in Euclidean spaces\footnote{Although all results we provide are generalizable to complete Riemannian manifolds}.  Let$\left\langle \cdot,\cdot\right\rangle$ be the canonical Euclidean scalar product on $\RR^p$ and $\|\cdot\|$ its associated norm. Take $p\in \NN$. A locally Lipschitz continuous function, $f \colon \RR^p \to \RR$ is differentiable almost everywhere by Rademacher's theorem, see for example \cite{evans2015measure}. Denote by $R \subset \RR^p$, the full measure set where $f$ is differentiable, then the Clarke subgradient of $f$ is given for any $x \in \RR^p$, by
\begin{align*}
				\partialc f(x) = \mathrm{conv} \left\{ v \in \RR^p,\, \exists  y_k \underset{k \to \infty}{\to} x \text{ with } y_k \in R,\, v_k = \nabla f(y_k) \underset{k \to \infty}{\to} v \right\}.
\end{align*}
A set valued map $D\colon \RR^p \rightrightarrows \RR^q$ is a function from $\RR^p$ to  the set of subsets of $\RR^q$. The graph of $D$ is given by
\begin{align*}
				\mathrm{graph}\,D = \left\{ (x,z):\, x \in \RR^p,\, z \in D(x)\right\}.
\end{align*}
$D$ is said to have {\em closed graph} or to be {\em graph closed} if $\mathrm{graph}\, D$ is closed as a subset of $\RR^{p + q}$. An equivalent characterization is that for any converging sequences $\left( x_k \right)_{k\in \NN}$, $\left( v_k \right)_{k\in \NN}$ in $\RR^p$, with $v_k \in D(x_k)$ for all $k \in \NN$, we have
\begin{align*}
				\lim_{k \to \infty} v_k \in D(\lim_{k\to\infty}x_k).
\end{align*}
An {\em absolutely continuous curve} is a continuous function $x \colon \RR \to \RR^p$ which admits a derivative $\dot{x}$ for Lebesgue almost all $t \in \RR$, (in which case $\dot{x}$ is Lebesgue measurable), and $x(t) - x(0)$ is the Lebesgue integral of $\dot{x}$ between $0$ and $t$ for all $t \in \RR$. Absolutely continuous curves are well suited to generalize differential equations to differential inclusions  \cite{aubin1984differential}. Given a set valued map $D\colon \RR^p \rightrightarrows \RR^p$, $x_0 \in \RR^p$, $x \colon \RR \to \RR^p$ is {\em a solution to the differential inclusion problem}
\begin{align*}
				&  \dot{x}(t) \in D(x(t)) \\
			& 	x(0) = x_0,
\end{align*}
if $x$ is an absolutely continuous curve satisfying $x(0) = x_0$ and $\dot{x}(t) \in D(x(t))$ for almost all $t$.

\subsection{Definition and vanishing circulations}
Throughout this section, we denote by $D \colon \RR^p \rightrightarrows \RR^p$ a set valued map with closed graph and nonempty compact values. The following lemma is derived from results presented in the overview textbook \cite[Section 18]{aliprantis2005infinite}.

\begin{lemma}
                Let $D \colon \RR^p \rightrightarrows \RR^p$ be a set valued map with non\-empty compact values and closed graph. Let $\gamma\colon [0,1] \to \RR^p$ be an absolutely continuous path. Then the following function
                \begin{align*}
                    t &\mapsto \max_{v \in D(\gamma(t))} \left\langle \dot{\gamma}(t), v \right\rangle, 
                \end{align*}
                defined almost everywhere on $[0,1]$, is Lebesgue measurable.
        \label{th:integrability}
\end{lemma}
\begin{proof}
				The set valued map  $t\rightrightarrows D(\gamma(t))$ has a closed graph by continuity of $\gamma$, and nonempty compact values, hence by \cite[Theorem 18.20]{aliprantis2005infinite} it is measurable in the sense of \cite[Definition 18.1]{aliprantis2005infinite}. Using \cite[Theorem 18.19]{aliprantis2005infinite}, for each $u \in \RR^p$ the function $t \mapsto \max_{v \in D(\gamma(t))} \left\langle v, u \right\rangle$ is Borel measurable. Further, the function
				$x \mapsto  \max_{v \in D(\gamma(t))} \left\langle x,v \right\rangle$ is continuous. As a consequence, the mapping $(t,x)\mapsto \max_{v \in D(\gamma(t))} \left\langle x,v \right\rangle$ is a Carath\'eodory integrand, and hence is jointly Borel measurable \cite[Lemma 4.51]{aliprantis2005infinite}. Since $\dot{\gamma}$ is Lebesgue measurable, we deduce that
$t\mapsto \max_{v \in D(\gamma(t))} \left\langle \dot{\gamma}(t),v \right\rangle$ is Lebesgue measurable.
\end{proof}

We can now proceed with the definition of a conservative set valued field.
\begin{definition}[Conservative set valued fields]{\rm 
                Let $D \colon \RR^p \rightrightarrows \RR^p$ be a set valued map. $D$ is a {\em conservative (set valued)  field} whenever it has closed graph, nonempty compact values and for any absolutely continuous loop $\gamma \colon [0,1] \to \RR^p$, that is   $\gamma(0) = \gamma(1)$, we have
                \begin{align*}
                    \int_0^1 \max_{v \in D(\gamma(t))} \left\langle \dot{\gamma}(t), v \right\rangle \dd t = 0
                \end{align*}
                where the integral is  understood in the Lebesgue sense\footnote{which is possible thanks to Lemma  \ref{th:integrability}.}. It is equivalent to require \begin{align*}
                    \int_0^1 \min_{v \in D(\gamma(t))} \left\langle \dot{\gamma}(t), v \right\rangle \dd t = 0
                \end{align*}
                for all loops $\gamma$.
        \label{def:conservative}
        }
\end{definition}
\begin{remark}[min, max circulations and conservativity] \label{rem:min}{\rm 
The min formula is indeed obtained by using the  the reverse path $\tilde{\gamma}(t) = \gamma(1-t):$ 
\begin{align*}
    \int_0^1 \max_{v \in D(\tilde{\gamma}(t))} \left\langle \dot{\tilde{\gamma}}(t), v \right\rangle \dd t & = \int_0^1 \max_{v \in D(\gamma(1-t))} \left\langle -\dot{\gamma}(1-t), v \right\rangle \dd t\\
    &= - \int_0^1 \min_{v \in D(\gamma(1-t))} \left\langle \dot{\gamma}(1-t), v \right\rangle \dd t\\
    &=\int_1^0 \min_{v \in D(\gamma(t))} \left\langle \dot{\gamma}(t), v \right\rangle \dd t\\
    &=-\int_0^1 \min_{v \in D(\gamma(t))} \left\langle \dot{\gamma}(t), v \right\rangle \dd t = 0.
\end{align*}
We deduce that for almost all $t \in [0,1]$, $\max_{v \in D(\gamma(t))} \left\langle \dot{\gamma}(t), v \right\rangle = \min_{v \in D(\gamma(t))} \left\langle \dot{\gamma}(t), v \right\rangle$.    }          
\end{remark}
\begin{remark}[Vanishing circulation and conservativity] {\rm
    The\-re is a measurable arg\-max selection in Lemma \ref{th:integrability} (see \cite[Theorem 18.19]{aliprantis2005infinite}) so that for any measurable selection $v \colon [0,1] \to \RR^p$, $v(t) \in D(\gamma(t))$ for all $t$, we have $\int_0^1 \left\langle \dot{\gamma}(t), v(t) \right\rangle \dd t = 0$. Thus, in the setting of Definition \ref{def:conservative}, an equivalent characterization is that the Aumann integral of $t \rightrightarrows \left \langle \dot{\gamma}(t), D(\gamma(t)) \right \rangle$ is $\{0\}$. In short
    \begin{equation}\label{circul0}
 \int_0^1 \left\langle  D(\gamma(t)), \dot{\gamma}(t)\right\rangle \dd t = \{0\},  
    \end{equation}
    exactly means  that $D$ is conservative. We recover the standard definition of conservativity as fields with vanishing  circulation.}
    \label{rem:aumannIntegral}
\end{remark}

\subsection{Locally Lipschitz continuous potentials of conservative fields}
\begin{definition}[Potential functions of conservative fields]{\rm
			 Let $D \colon \RR^p \rightrightarrows \RR^p$ be a conservative field. A function $f$ defined through any of the  equivalent forms \begin{eqnarray}
			f(x)& = & f(0)+\int_0^1 \max_{v \in D(\gamma(t))} \left\langle \dot{\gamma}(t), v \right\rangle \dd t \label{pot1}\\
			     &= & f(0)+\int_0^1 \min_{v \in D(\gamma(t))} \left\langle \dot{\gamma}(t), v \right\rangle \dd t \label{pot2}\\
			     & = & f(0)+\int_0^1 \left \langle \dot{\gamma}(t), D(\gamma(t)) \right\rangle \dd t
			\label{pot3}\end{eqnarray}
			for any $\gamma$ absolutely continuous with $\gamma(0) = 0$ and $\gamma(1)=x$. $f$ is well and uniquely defined up to a constant. It is called a {\em potential function for~$D$}. We shall also say that {\em $D$ admits $f$ as a potential,}
				  or that {\em $D$ is a conservative field for~$f$}.
		\label{def:conservativeMapForF}}
\end{definition}
\begin{remark}{\rm
(a) To see that the definitions \eqref{pot1}, \eqref{pot2}  and \eqref{pot3} are indeed equivalent and independent of the chosen path, one adapts   classical ideas as follows.  Consider any $x \in \RR^p$, and any absolutely continuous paths $\gamma_1$, $\gamma_2$ such that $\gamma_1(0) = \gamma_2(0) = 0$ and $\gamma_1(1) = \gamma_2(1)=x$. We have
    \begin{align*}
        &\int_0^1 \max_{v \in D(\gamma_1(t))} \left\langle \dot{\gamma_1}(t), v \right\rangle \dd t - \int_0^1 \min_{v \in D(\gamma_2(t))} \left\langle \dot{\gamma_2}(t), v \right\rangle \dd t \\
         = &\int_0^1 \max_{v \in D(\gamma_1(t))} \left\langle \dot{\gamma_1}(t), v \right\rangle \dd t + \int_0^1 \max_{v \in D(\gamma_2(t))} -\left\langle \dot{\gamma_2}(t), v \right\rangle \dd t \\
        = &\  \int_0^{\frac{1}{2}} \max_{v \in D(\gamma_1(2t))} \left\langle 2\dot{\gamma_1}(2t), v \right\rangle \dd t + \int_{\frac{1}{2}}^1 \max_{v \in D(\gamma_2(2-2t ))} \left\langle-2 \dot{\gamma_2}(2-2t), v \right\rangle \dd t\\
        =&\ 0
    \end{align*}
    since the concatenation of $t \mapsto \gamma_1(2t)$ for $0 \leq t \leq 1/2$ and $t \mapsto \gamma_2(2-2t)$ for $1/2 \leq t \leq 1$ is an absolutely continuous loop. This shows that the value of the integral does not depend on the path. The ``minimum and maximum integrals" are thus equal and we may set for any $x \in \RR^p$:
    \begin{align*}
						f(x) = f(0) + \int_0^1 \max_{v \in D(\gamma(t))} \left\langle \dot{\gamma}(t), v \right\rangle \dd t = f(0) + \int_0^1 \min_{v \in D(\gamma(t))} \left\langle \dot{\gamma}(t), v \right\rangle \dd t
    \end{align*}
    for any $\gamma$ absolutely continuous with $\gamma(0) = 0$ and $\gamma(1)=x$. The right hand-side in \eqref{pot3} is thus a single number, and the identity is therefore well defined.\\
      (b) If $f$ is differentiable, $\nabla f$ is of course a conservative field (it is not unique).  More examples and a discussion are provided in Subsection \ref{sec:differentiability}.\\
    (c) The definition can be directly extended to star-shaped domains.\\
    (d) The potential function $f$ is  locally Lipschitz continuous. Indeed take a bounded set $S$. Take $x,y\in S$ and use (c) above with the path $[0,1]\ni t\to \gamma(t)=tx+(1-t)y$, 
    $$|f(y)-f(x)|\leq \|y-x\| \int_0^1\max_{v \in D(\gamma(t))}  \|v\|  \dd t\leq M\|x-y\|$$
   where $M$ is such that $$M\geq \max\{\|v\|: x \in \overline{\conv\, S},v \in D(x)\}$$
   with $\conv \, S$ being the convex envelope of $S$. From \cite[Lemma 3]{borwein2001generalized}, $D$ is locally bounded and such a finite constant must exist.\\
   (e) If $D_1,D_2$ are two graph closed set valued mappings with compact nonempty values, then $D_1\subset D_2$ and $D_2$ conservative implies that $D_1$ is conservative as well. \\Observe also that if $D$ is conservative $x \rightrightarrows \mathrm{conv}(D(x))$ is conservative as well.
   
  }
  \end{remark} 
  
\smallskip

Chain rule characterizes conservativity in the following sense:
	\begin{lemma}[Chain rule and conservativity]
				Let $D \colon \RR^p \rightrightarrows \RR^p$ be a locally bounded, graph closed set valued map and $f\colon \RR^p \to \RR$ a locally Lipschitz continuous function. Then $D$ is a conservative field for $f$, if and only if for any absolutely continuous curve $x \colon [0,1] \to \RR^p$, the function $t \mapsto f(x(t))$ satisfies 
				\begin{align}
								\frac{d}{d t} f(x(t)) = \left\langle v, \dot{x}(t) \right\rangle\qquad \forall v \in D_f(x(t)),
								\label{eq:chainRuleAC}
				\end{align}
				for almost all $t \in [0,1]$.
				\label{lem:chainRuleAC}
\end{lemma}

\begin{proof} 
				The reverse implication is obvious, using Lemma \ref{th:integrability}, integrating the characterization in equation \eqref{eq:chainRuleAC} we obtain any of the equivalent equations of Definition \ref{def:conservativeMapForF}.
				To prove the converse, assume now that $D$ is a conservative field for $f$. For any $0<s<1$, we have
    \begin{align*} 
          f(x(s)) - f(x(0))& 
         = \int_0^1 \max_{v \in D(x(st))} \left\langle s\dot{x}(st), v \right\rangle \dd t   \\
         &= \int_0^s \max_{v \in D(x(t))} \left\langle \dot{x}(t), v \right\rangle \dd t \\
         &=\int_0^s \min_{v \in D(x(t))} \left\langle \dot{x}(t), v \right\rangle \dd t
    \end{align*}
    The fundamental theorem of calculus states that $s \mapsto f(x(s))$ is differentiable almost everywhere and for almost all $s \in [0,1]$,
    \begin{align*}
        \frac{d}{ds} f(x(s)) = \max_{v \in D(x(s))} \left\langle \dot{x}(s), v \right\rangle = \min_{v \in D(x(s))} \left\langle \dot{x}(s), v \right\rangle = \left\langle \dot{x}(s), v \right\rangle,
    \end{align*}
    for all $v \in D(x(s))$.  
\end{proof}

\section{A generalized differential calculus}

\subsection{Conservativity,  Clarke subdifferential and gradient a.e.}

We start with the following fundamental result.

\begin{theorem}[Conservative fields are gradients almost everywhere]\label{t:Rade}
			Consider a \\con\-ser\-va\-tive field $D \colon \RR^p \rightrightarrows \RR^p$ for $f \colon \RR^p \to \RR$. Then
    $D=\{\nabla f\}$ Lebesgue almost everywhere.
    \label{lem:gradientAlmostEverywhere}
\end{theorem}

\begin{proof}
				Fix a measurable selection $a \colon \RR^p \to \RR^p$ of $D$ and $f$ a potential for $D$ (which is locally Lipschitz continuous by definition). Measurable selections exist because $D$ has closed graph, with nonempty values, and hence is measurable in the sense of \cite[Definition 18.1]{aliprantis2005infinite} so that \cite[Theorem 18.13]{aliprantis2005infinite} applies (see also Lemma \ref{th:integrability}). Fix a direction $v \in \RR^p$, $x \in \RR^p$ a base point, let $s<t$ be real numbers and $\gamma$ the path $\gamma(\tau)=(1-\tau)(x+sv)+ \tau (x+tv)$, then by using  conservativity and  an elementary change of variable, we obtain for all $x \in \RR^p$ 
$$f(x+tv)-f(x+sv)=\int_s^t\langle a(x+\tau v),v\rangle d\tau.$$
    Using the fundamental theorem of calculus (in its Lebesgue form), one obtains 
		\begin{align}
						f'(y;v) =\langle a(y),v\rangle
						\label{eq:equalityOnALine}
		\end{align}
		almost everywhere on the line $x+\R v$, where  $$f'(y;v):=\lim_{r\to 0,\,r\neq 0} \frac{f(y+rv)-f(y)}{r}$$ when the limit exists. 
		Since $f$ is continuous, the following two functions defined for all $y \in \RR^p$:
\begin{align*}
				f_u'(y;v):=\lim\sup_{s\to 0,\,s\neq 0} \frac{f(y+sv)-f(y)}{s} = \lim_{k \to \infty} \sup_{0 < |s| \leq 1/k} \frac{f(y+sv)-f(y)}{s}\\
				f_l'(y;v):=\lim\inf_{s\to 0,\,s\neq 0} \frac{f(y+sv)-f(y)}{s} = \lim_{k \to \infty} \inf_{0 < |s| \leq 1/k} \frac{f(y+sv)-f(y)}{s},
\end{align*}
are Borel, hence Lebesgue measurable. Consider the following set 
\begin{align*}
				A = \left\{ y \in \RR^p, \, f_u'(y;v) \neq \langle v,a(y)\rangle \text{ or } f_l'(y;v) \neq \langle a(y), v\rangle \right\}. 
\end{align*}
This set is Lebesgue measurable and for any $y \in \RR^p \setminus A$, we have
\begin{align*}
				f'(y;v) = f_u'(y;v) = f_l'(y;v) = \langle a(y),v\rangle.
\end{align*}
Furthermore, using \eqref{eq:equalityOnALine} we have $\HH^1(A \cap (x + \RR v)) = 0$, where $\HH^1$ is the Hausdorff measure of dimension $1$. Since $x \in \RR^p$ was arbitrary, we actually have $\HH^1(A \cap L)=0$ for any line $L$, parallel to $v$ and since $A$ is measurable, Fubini's theorem entails that $A$ has zero Lebesgue measure, and hence we have $f'(y;v) =\langle v,a(y)\rangle$ for almost all $y \in \RR^p$. Now the Rademacher Theorem \cite[Theorem 3.2]{evans2015measure}, ensures that $f$ is differentiable almost everywhere, this implies that $f'(y;v) = \left\langle \nabla f (y), v\right\rangle$ for almost all $y \in \RR^p$ and hence, $\left\langle \nabla f (y), v\right\rangle = f'(y;v) = \langle a(y),v\rangle$ for almost all $y \in \RR^p$ . 

The direction $v$ was chosen arbitrarily, we repeat the same construction for each $v$ among $p$ elements of the canonical basis of $\RR^p$. Since the union of $p$ Lebesgue null sets has zero measure, we obtain that $a(y) = \nabla f(y)$ for almost all $y \in \RR^p$. 

Since $a$ was chosen as an arbitrary measurable selection for $D$, we may use \cite[Corollary 18.15]{aliprantis2005infinite} which states that there is a sequence of measurable selections for $D$, $(a_k)_{k \in \NN}$ such that for any $x \in \RR^p$, $D(x) = \mathrm{cl}\ \{a_i(x)\}_{i \in \NN}$. Using the previous  Rademacher's argument for each $i$ in $\NN$, there exists a sequence of measurable sets $(S_i)_{i \in \NN}$ which  all have full measure  and such that $a_i = \nabla f$ on $S_i$. Setting $\displaystyle S = \cap_{i \in \NN} S_i$, we have that $\R^p\setminus S$ has zero measure and 
$a_i = \nabla f$ on $S$ for all $i$ in $\NN$ and hence, using \cite[Corollary 18.15]{aliprantis2005infinite}, $D = \{\nabla f\}$ on $S$. This proves the desired result.
\end{proof}

An important consequence of the above result is that Clarke subdifferential appears as a minimal conservative field among convex valued conservative fields.
\begin{corollary}[Clarke subgradient as a minimal convex conservative field]  \label{cor:clarkeMinimalConservative}
    Let\\ $f \colon \RR^p \to \RR$  admitting a conservative field $D \colon \RR^p \rightrightarrows \RR^p$. Then $\partialc f$ is a conservative field for $f$, and for all $x \in \RR^p$,
    \begin{align*}
        \partialc f (x) \subset \mathrm{conv}(D(x)).
    \end{align*}
		\label{cor:clarkeMinimal}
\end{corollary}
\begin{proof} Let $S \subset \RR^p$ be a full measure set such that $D = \nabla f$ on $S$ (such a set exists by Theorem \ref{lem:gradientAlmostEverywhere}). Using \cite[Theorem 2.5.1]{clarke1983optimization}, we have, for any $x \in \RR^p$
    \begin{align*}
        \partialc f(x) = \mathrm{cl}\ \mathrm{conv} \left( \left\{ \lim_{k \to \infty} \nabla f(x_k),\, x_k \in S, \, x_k \underset{k \to \infty}{\to} x\right\}\right).
    \end{align*}
    Since $D$ has closed graph and $D = \nabla f$ on $S$, we have
    \begin{align*}
        \mathrm{cl}\ \mathrm{conv} \left( \left\{ \lim_{k \to \infty} \nabla f(x_k),\, x_k \in S, \, x_k \underset{k \to \infty}{\to} x\right\}\right) 
        \subset \mathrm{cl}\ \mathrm{conv} \left( D(x)\right) = \mathrm{conv} \left( D(x)\right),
    \end{align*}
    which allows to conclude. The fact that $\partialc f$ is conservative, follows right from the definition and the previous inclusion.
\end{proof}

We deduce from Corollary \ref{cor:clarkeMinimalConservative} a Fermat's rule for conservative fields.

\begin{proposition}[Fermat's rule]\label{t:fermat}
                Let $f \colon \RR^p \to \RR$ be a potential for $D \colon \RR^p \rightrightarrows \RR^p$ with non\-empty com\-pact values and closed graph. Let $x \in \RR^p$ be a local minimum or local maximum of $f$. Then $0 \in \mathrm{conv}(D(x))$.
                
\end{proposition}
\begin{proof}
				This is a consequence of Corollary \ref{cor:clarkeMinimalConservative} since Fermat's rule holds for the Clarke subdifferential \cite[Proposition 2.3.2]{clarke1983optimization}.
\end{proof}

Given a fixed conservative field $D$ with $f$ as a potential, we say that $x$ is  {\em $D$-critical} for $f$ if $D(x)\ni 0.$ The value $f(x)$ is then called {\em a $D$-critical value.} This idea originates in \cite{castera2019inertial}.

\begin{remark}{\em 
    The convex envelope in Fermat's rule is necessary. For example, let $D \colon x \rightrightarrows \mathrm{sign}(x)$ with $D(0) = \{-1,1\}$, then $D$ has closed graph and is conservative for the absolute value. The origin is a global minimum of the potential but $0 \not\in D(0)$.}
\end{remark}

\subsection{Path differentiability}
\label{sec:differentiability}

Conservative fields convey a natural notion of  ``generalized differentiability", a function being differentiable if it admits a conservative field for which Definition \ref{def:conservativeMapForF} holds true. We call such functions 
path differentiable and provide a characterization in this section.
\begin{definition}[Path differentiability]{\rm
We say that  $f: \RR^p \to \RR$ is {\em path differentiable} if $f$ is the potential of a conservative field on $\RR^p$.}
\end{definition}
This implies that $f$ is locally Lipschitz continuous. We deduce from Corollary \ref{cor:clarkeMinimalConservative} the following  characterization of path differentiable functions.
\begin{corollary}[Characterization of path differentiable functions]    \label{cor:clarkeCharacterization}
				Let $f \colon \RR^p \to \RR$ be locally Lipschitz continuous, then the following assertions are equivalent\begin{itemize}
				    \item[(i)] $f$ is path differentiable
				    \item[(ii)] $\partial^cf$ is a conservative field  
				    \item[(iii)] $f$ has chain rule for the Clarke subdifferential.
				\end{itemize} 
\end{corollary}
\begin{proof}
				The items {\it(ii)} and {\it(iii)} are equivalent  by Lemma~\ref{lem:chainRuleAC}; we just need to prove $\mbox{\it (i)}\iff\mbox{\it{(ii)}}$. The direct implication is a consequence of Corollary \ref{cor:clarkeMinimalConservative}. For the reverse implication, we assume that $\partial^c f$ is conservative and we  show that $f$ is a potential for $\partial^c f$. Let $\gamma \colon [0,1] \to \RR^p$ be an absolutely continuous path. Using Remark \ref{rem:min}, we have for almost all $t \in [0,1]$,
				\begin{align}
								\min_{v \in \partial^c f(\gamma(t))} \left\langle v, \dot{\gamma}(t) \right\rangle = \max_{v \in \partial^c f(\gamma(t))} \left\langle v, \dot{\gamma}(t) \right\rangle
								\label{eq:minmax}
				\end{align}
				This implies that the affine span of $\partial^c f(\gamma(t))$ is orthogonal to $\dot{\gamma}(t)$ for almost all $t$. This corresponds to the definition of a ``saine function" introduced in  \cite{valadier1989entrainement}.  Proposition 3 in \cite{valadier1989entrainement} shows that the value in \eqref{eq:minmax} is equal to $\frac{d}{dt} f(\gamma(t))$ for almost all $t$. This is the equivalent characterization for a potential of $\partial^c f$ in Lemma \ref{lem:chainRuleAC}, hence $f$ is a potential for $\partial^c f$. 
\end{proof}

The following property is sometimes called integrability, it has been studied for convex functions in \cite{rockafellar1970maximal} and for broader classes in, e.g.,  \cite{correa1989tangentially,thibault1995integration,borwein1997essentially,thibault2005integrability}.
\begin{corollary}[Integrability and Clarke subdifferential]
Let $f \colon \RR^p \to \RR$ and $g \colon \RR^p \to \RR$ be two locally Lipschitz path differentiable functions such that $\partialc g(x) \subset \partialc f(x)$ for all $x \in \RR^p$, then $f-g$ is constant.
\end{corollary}
\begin{proof}
    Path differentiability entails that subdifferentials are conservative fields by Co\-rol\-la\-ry \ref{cor:clarkeMinimalConservative}. The result follows by definition of potential functions through integration in Definition \ref{def:conservativeMapForF}.
\end{proof}

We briefly compare some standard subgradients notions and conservative fields. We use the vocabulary and notation of \cite{rockafellar1998Variational}.

\begin{proposition}[Some path differentiable functions]
 Let $f:\R^p\to \R$ be Lipschitz continuous, the following are sufficient conditions for $f$ to be path differentiable
 
 (i) $f$  is convex or concave.
 
 (ii) $f$ or $-f$  is Clarke regular.

 (iii) $f$ or $-f$  is prox regular.
 
 (iv) $f$ is real semialgebraic (or more generally tame, i.e.,  definable in some o-minimal structure).
    
\end{proposition}
\begin{proof}
				Using the chain rule characterization, all proofs boil down to providing a chain rule with the Clarke subdifferential for each of the above mentioned situation. We refer to \cite{rockafellar1998Variational} for convex, Clarke and prox regular functions, \cite{davis2018stochastic} for definable functions. 
\end{proof}

In general, conservative fields may be distinct from all other classical subdifferentials, even in the definable case. Define for instance $D:\R\to\R$ by $D(0)=\{-1,0,1\}$, $D(1)=[0,2]$ and $D(x)=0$ otherwise. It is a conservative field on $\R$ with any constant function as a potential function.

\begin{remark}[Historical aspects]{\rm  Our effort to define a subclass of locally Lipschitz continuous functions which has favorable differentiability properties is one attempt among many others. The closest idea we found is due to  Valadier who introduced in 1989 the notion of \textit{``fonctions saines"} \cite{valadier1989entrainement}. Although Definition \ref{def:conservativeMapForF} looks much more general than the notion given in \cite{valadier1989entrainement}, the equivalent characterization of Corollary \ref{cor:clarkeCharacterization} shows that path-differentiable and ``saines" functions are actually the same.  
Later on, at the end of the nineties, Borwein and Moors introduced the notion of essentially smooth functions (strictly differentiable almost everywhere) as a  well-behaved subclass of locally Lipschitz continuous functions \cite{borwein1997essentially}. Interestingly, the notion of saines functions was as well reconsidered and slightly modified in \cite{borwein1998chain} to describe the larger class of arcwise essentially smooth functions.  Following \cite{valadier1989entrainement} and  Chapter 1 of \cite{wang1995pathological}, we see that, in the univariate case, saine and essentially smooth functions coincide. This is no longer true for $p\geq2$,  the set of ``fonctions saines" is a strict subset of essentially smooth functions. 
}\end{remark}

\begin{remark}[Genericity: theory and practice]
{\rm      
The work of Wang et al. \cite{wang1995pathological,borwein2001generalized} 
allows to claim that generic $1$-Lispchitz functions are not path differentiable. Paradoxi\-cally, we shall see in further sections that most functions arising in applications are path differentiable (e.g., any semialgebraic or definable  function is path differentiable)\footnote{Valadier's terminology finds here a surprising  justification, since  ``saine", healthy in English, is chosen as an antonym to ``pathological"}. }
          
\end{remark}

\subsection{Conservative mappings and their calculus}
 
In this part we often identify linear mappings to their matrices in the canonical basis. For general conservative mappings, we adopt here a definition through the chain rule rather than circulations in order to simplify the exposition. However it would be relevant to provide a direct extension of Definition \ref{def:conservative} involving vanishing circulations through set valued integration, this is matter for future work.

\begin{definition}[Conservative mappings]	\label{def:DIjacobian}{\rm 
				Let $F \colon \RR^p \to \RR^m$ be a locally Lipschitz function.  $J_F \colon \RR^p \rightrightarrows \RR^{m \times p}$ is called a {\em conservative mapping} for $F$, if for any absolutely continuous curve $\gamma \colon [0,1] \to \RR^p$, the function $t \mapsto F(\gamma(t))$ satisfies
				\begin{align*}
								\frac{d}{d t} F(\gamma(t)) = V \dot{\gamma}(t) \qquad \forall V \in J_F(\gamma(t))
				\end{align*}
				for almost all $t \in [0,1]$.}
\end{definition}
\begin{remark}[Conservative fields are conservative mappings]                \label{rem:conservativeVectorMap}{\rm  If $D \colon \RR^p \rightrightarrows \RR^p$ is a conservative field for $f \colon \RR^p \to \RR$, it is of course also a conservative mapping for~$f$.}
\end{remark}

The following lemma  provides an elementary but essential way to construct conservative matrices.
\begin{lemma}[Componentwise aggregation]	\label{lem:DIjacobian}
				Let $F \colon \RR^p \to \RR^m$ be a locally Lipschitz continuous function. Let $J_F \colon \RR^p \rightrightarrows \RR^{m \times p}$ be given by:
         				\begin{align*}
         				    J_F(x) &= 
         				    \left(
         				        \begin{tabular}{c}
         				             $v_1^T$ \\
         				            $\vdots$ \\
         				            $v_m^T$
         				        \end{tabular}
         				    \right), \qquad
         				    v_i \in D_i(x), \,i = 1\ldots, m, \quad \forall x \in \RR^p,
         				\end{align*}
                where $D_i$ is a conservative field for the $i$-th coordinate of $F$, $i = 1, \ldots, m$. Then $J_F$ is a conservative mapping for $F$.
\end{lemma}
\begin{proof}
                This follows from Lemma \ref{lem:chainRuleAC}, the product structure of $J_F$ and the fact that a finite union of Lebesgue null sets is a Lebesgue null set.
\end{proof}

A partial converse holds true (thanks to Lemma \ref{lem:chainRuleAC}): projection on rows of conservative mappings have to be conservative mappings for the corresponding coordinate function.

\begin{lemma}[Coordinates of conservative mappings]    \label{lem:projRowJacobian}
    Let $F \colon \RR^p \to \RR^m$ be locally Lipschitz continuous. Let $J_F \colon \RR^p \rightrightarrows \RR^{m \times p}$ be a conservative  mapping for $F$, then the projection of $J$ on the first row of $J$, is a conservative field for the first coordinate of~$F$.
\end{lemma}

Observe however, that these ``generalized Jacobians" may have a more complex structure than the product structure outlined in Lemma \ref{lem:DIjacobian}.

The following chain rule of generalized differentiation follows readily from the definition.

\begin{lemma}[The product of conservative mappings is conservative]	\label{lem:compositionJacobian}
				 Let $F_1 \colon \RR^p \to \RR^m$, $F_2 \colon \RR^m \to \RR^l$ be locally Lipschitz continuous mappings and $J_1 \colon \RR^p \rightrightarrows \RR^{m \times p}$,  $J_2 \colon \RR^p \rightrightarrows \RR^{l \times m}$ be some associated conservative mappings. Then the product mapping $J_2 J_1$ is a conservative mapping for $F_2 \circ F_1$.
\end{lemma}
\begin{proof} Consider any absolutely continuous curve $\gamma \colon [0,1] \to \RR^p$. By local Lipschitz continuity, $t \mapsto F_1(\gamma(t))$ is also absolutely continuous and by definition of $J_1$ we have, 
				\begin{align*}
								\frac{d}{d t} F_1(\gamma(t)) = J_1(\gamma(t)) \dot{\gamma}(t) \qquad \mbox{a.e. on }(0,1).
				\end{align*}
				Furthermore, $F_2\circ F_1 \circ \gamma$ is also absolutely continuous by the local Lipschitz continuity of $F_2$. From the definition of $J_2$, we have
				\begin{align*}
								 \frac{d}{d t}\left( F_2 \circ F_1(\gamma(t)) \right)= J_2(F_1(\gamma(t))) \times \frac{d}{d t} \left( F_1(\gamma(t)) \right) \qquad \mbox{ a.e. on }(0,1)
				\end{align*}
				The last two identities lead to the conclusion.
\end{proof}

We deduce the following chain rule by enlargement.
\begin{lemma}[Outer chain rule]
				Let $F \colon \RR^p \to \RR^m$ and $g \colon \RR^m \to \RR$ be locally Lipschitz continuous. Let $D_F \colon \RR^p \rightrightarrows \RR^{m \times p}$ and $D_g\colon \RR^m \rightrightarrows \RR^m$ be some set valued  mappings such that $F_i$, the $i$-th coordinate of $F$, is a potential for $[D_F(\gamma(t))]_i$, the $i$-th row of $D_F$ for $i = 1, \ldots, m$ and  $g$ is a potential for $D_g$. Then $g\circ F$ is a potential of $D \colon x \rightrightarrows D_F(x)^T D_g(F(x))$.
				\label{lem:compositionSubgradient}
\end{lemma}
\begin{proof}
                This is obtained combining Lemmas \ref{lem:DIjacobian}, \ref{lem:compositionJacobian} and Remark \ref{rem:conservativeVectorMap}.
\end{proof}

   A simple consequence is a ``sum rule by enlargement" which is fundamental in the study of the mini-batch stochastic gradient:
\begin{corollary}[Outer sum rule]\label{c:sumRule}
 Let $f_1,\ldots, f_n$ be locally Lipschitz continuous functions. Then $f = \sum_{i=1}^n f_i$ is a potential for $D_f = \sum_{i=1}^n D_i$ provided that $f_i$ is a potential for each $D_i$, $i=1,\ldots, n$.
\end{corollary}

\section{Tameness and conservativity}\label{s:tame}
Let us beforehand provide two useful reading keys:
\begin{itemize}
    \item[---] The reader unfamiliar with definable  objects can simply replace definability  assumptions by semialgebraicity assumptions. It is indeed enough to treat major   applications considered here, as for example deep learning with ReLU activation functions and square loss.
 \item[---] Semialgebraicity and definability being easy to recognize in practice, the results in this section can be readily used as ``black boxes" for applicative purposes. 
\end{itemize}

\subsection{Introduction and definition}

We recall here the  results of  geometry  that we use in the present work.
Some references on this topic are \cite{Cos99,dries1996geometric}.

  An {\em o-minimal structure} on $(\R,+,\cdot)$ is a collection of sets
  $\mathcal{O} = (\mathcal{O}_p)_{p \in \NN}$ where each $\mathcal{O}_p$ is itself a family of
  subsets of $\R^p$, such that for each $p \in \NN$:
  \begin{enumerate}
  \item[(i)] $\mathcal{O}_p$ is stable by complementation, finite union, finite intersection and contains $\R^p$.
    \item[(ii)]  if $A$ belongs to $\mathcal{O}_p$, then $A \times \R$ and $\R \times A$
      belong to $\mathcal{O}_{p+1}$;
    \item[(iii)]  if $\pi: \R^{p+1} \to \R^p$ is the canonical projection onto $\R^p$ then,
      for any $A \in \mathcal{O}_{p+1}$, the set $\pi(A)$ belongs to $\mathcal{O}_p$;
      \label{it:algebraic}
    \item[(iv)]  $\mathcal{O}_p$ contains the family of real algebraic subsets of $\R^p$, that is,
      every set of the form
      \[
        \{ x \in \R^p \mid g(x) = 0 \}
      \]
      where $g: \R^p \to \R$ is a polynomial function;
    \item[(v)]  the elements of $\mathcal{O}_1$ are exactly the finite unions of  intervals.
  \end{enumerate}

A subset of $\R^p$ which belongs to an o-minimal structure $\mathcal{O}$ is said to be
{\em definable in $\mathcal{O}$}. The terminology {\em tame} refers to definability in an o-minimal structure without specifying which structure. Very often the o-minimal structure is fixed, so one simply says {\em definable}. We will stick to this terminology all our results are valid in a fixed o-minimal structure.
A set valued mapping (or a function) is said to be definable in $\mathcal{O}$ whenever its graph is definable in $\mathcal{O}$.

  The simplest  o-minimal structure is given by the class of
  real  semialgebraic objects.
  Recall that a set $A \subset \R^p$ is called {\em semialgebraic} if it is a finite union of sets of the form $$\displaystyle  \bigcap_{i=1}^k \{x \in \R^p \mid g_{i}(x) < 0, \; h_{i}(x) = 0 \}$$
  where the functions $g_{i}, h_{i}: \R^p \to \R$ are real polynomial functions and $k\geq 1$.
  The key tool to show that these sets form an o-minimal structure is Tarski-Seidenberg
  principle.

	O-minimality is an extremely rich topological concept: major structures, such as globally subanalytic sets or sets belonging to the $\log$-$\exp$
  structure provides vast applicative opportunities (as deep learning with hyperbolic activation functions or entropic losses, see \cite{davis2018stochastic,castera2019inertial} for some illustrations).
  We will not give proper definitions of these structures in this paper, but the interested
  reader may consult \cite{dries1996geometric}.

	The tangent space at a point $x$ of a manifold $M$ is denoted by $T_xM$. Given a submanifold\footnote{We only consider embedded submanifolds} $M$ of a finite dimensional Riemannian manifold, it is endowed by the Riemanninan structure inherited from the ambient space. Given $f \colon \RR^p \to \RR$ and $M\subset\R^p$ a differentiable submanifold  on which $f$ is differentiable, we denote by $\mbox{grad}_M f$ its Riemannian gradient or even, when no confusion is possible,  $\grad f$.

 A $C^r$ stratification of a (sub)manifold $M$ (of $\R^p$) is a partition $\SSS=(M_1,\ldots,M_m)$ of $M$ into $C^r$ manifolds having the property that $\cl M_i\cap M_j\neq \emptyset$ implies that $M_j$ is entirely contained in the boundary of $M_i$ whenever $i\neq j$. Assume that a function $f:M \to \R$ is given and that $M$ is stratified into manifolds on which $f$ is differentiable. For $x$ in $M$, we denote by $M_x$ the strata containing $x$ and we simply write $\grad f(x)$ for the gradient of $f$ with respect to $M_x$.

Stratifications can have many properties, we refer to \cite{dries1996geometric} and references therein for an account on this question and in particular for more on the idea  of a Whitney stratification that we will use repeatedly. 
We pertain here to one basic definition: a $C^{r}$-stratification $\SSS=(M_{i})_{i\in I}$ of a manifold $M$ has the
\emph{Whitney-}($a$)\emph{ property,} if for each $x\in\cl M_{i}\cap
M_{j}$ (with $i\neq j$) and for each sequence $(x_{k})_{k\in\NN}\subset
M_{i}$ we have:
\[
\left.
\begin{array}
[c]{ll}
& \underset{k\rightarrow\infty}{\lim}\mathcal{\;}x_{k}\mathcal{\;}=x\\
& \text{}  \\
&
\underset{k\rightarrow\infty}{\lim}\mathcal{\;}T_{x_{k}}M_{i}\mathcal{\;}
=\mathcal{T}
\end{array}
\right\}  \mathcal{\;}\Longrightarrow\mathcal{\;}T_{x}M_{j}\mathcal{\;}
\subset\mathcal{\;T}
\]
where the second limit is to be understood in the Grassmanian, i.e., ``directional", sense. In the
sequel we shall use the term \emph{Whitney stratification} to refer to a
$C^{1}$-stratification with the Whitney-($a$) property.

\subsection{Variational stratification and projection formulas}

Let us fix an o-minimal structure $\cal O$, so that a set or a function will be called definable if it  is definable in $\cal O$.
\begin{definition}[Variational stratification]{\rm 
				Let $f \colon \RR^p \to \RR$, be locally Lipschitz con\-ti\-nuous, let $D \colon \RR^p \rightrightarrows \RR^p$ be a set valued map and let $r\geq 1$. We say that the couple $(f,D)$ has a $C^r$ {\em variational stratification} if there exists a $C^r$ Whitney stratification $\SSS = (M_i)_{i \in I}$ of $\RR^p$, such that $f$ is $C^r$ on each stratum  and for all $x \in \R^p$,
				\begin{align}\label{projf}
								\mathrm{Proj}_{T_{M_x}(x)} D(x) = \left\{ \grad f(x) \right\},
				\end{align}				
				 where $\grad f(x)$ is the gradient of $f$ restricted to the active strata $M_x$ containing $x$.	}	\label{def:projectionFormula}
\end{definition}

The equations \eqref{projf} are called \emph{projection formulas} and are motivated by Corollary 9 in \cite{bolte2007clarke} which  states that Clarke subgradients of definable functions have projection formulas.
\begin{theorem}[Projection formula \cite{bolte2007clarke}] 
				Let $f \colon \RR^ p \to \RR$ be  definable, locally Lipschitz continuous\footnote{In \cite{bolte2007clarke} the authors assume $f$ to be arbitrary and obtain similar result, for simplicity we pertain to the Lipschitz case.}, and let $r \in \NN$. Then there exists a finite $C^r$ Whitney stratification $\SSS = (M_i)_{i \in I}$ of $\RR^p$ such that for all $x \in \RR^p$,
				\begin{align*}
								\mathrm{Proj}_{T_x(M_x)} \partialc f(x) = \left\{ \grad f(x) \right\}.
				\end{align*}
		In other words, the couple $(f,\partialc f)$ has a a $C^r$ variational stratification.
\label{th:stratification}
\end{theorem}

\subsection{Characterization of definable conservative fields}

The following is a  direct extension of the chain rule result  given in \cite[Theorem 5.8]{davis2018stochastic}. It relies also on   Theorem~\ref{th:stratification} and implies that a definable function $f$ is a potential of its Clarke subgradient.
\begin{theorem}[Integrability ({\rm from \cite{davis2018stochastic}})] Let $f \colon \RR^p \to \RR $, be locally Lispchitz continuous and let $D \colon \RR^p \rightrightarrows \RR^p$ be compact valued and graph closed, such that $(f,D)$ has a $C^1$ variational stratification. Then $f$ is a potential of $D$.
				\label{th:projImpliesChainRule}
\end{theorem}
\begin{proof}
    The proof given in \cite{davis2018stochastic} was proposed for the Clarke subdifferential but holds for larger classes of stratifiable set valued map. We reproduce the arguments here for clarity.
    
    Let $\SSS$ be a stratification provided by the variational stratification. Fix an absolutely continuous path $\gamma \colon [0,1] \to \RR^p$. Fix an arbitrary $t \in (0,1)$ and $M \in \SSS$ such that 
    \begin{align}
        \gamma(t) \in M,\quad \dot{\gamma}(t) \in T_M(\gamma(t)) 
        \label{eq:conditionTangentStrat}
    \end{align}
    In this case, the projection formula ensures that for any $v \in D(\gamma(t))$
    \begin{align*}
        \left\langle\dot{\gamma}(t),v \right\rangle = \left\langle\dot{\gamma}(t), \grad f(\gamma(t)) \right\rangle = \frac{d}{d t} f(\gamma(t)).
    \end{align*}
    Set $\Omega_M = \{t \in [0,1],\, \gamma(t) \in M,\quad \dot{\gamma}(t) \not\in T_M(\gamma(t)) \}$. Fix any $t_0 \in \Omega_M$, there exists a small closed interval $I$ centered at $t_0$ such that $I \cap \Omega_M = \{t_0\}$, otherwise one would have $ \dot{\gamma}(t_0) \in T_M(\gamma(t_0))$. The interval $I$ may be chosen with rational endpoints and this gives an injection from $\Omega_M$ to $\mathbb{Q}^2$. Hence $\Omega_M$ is countable. Since $\SSS$ contains only finitely many strata, for almost all $t \in [0,1]$, relation \eqref{eq:conditionTangentStrat} holds for some other strata and the result follows using the chain rule characterization of conservativity in Lemma \ref{lem:chainRuleAC}.
\end{proof}

We aim at proving the following converse in the context of tame analysis.
\begin{theorem}[Variational stratification for definable conservative fields]\label{t:var}
				Let $D \colon \RR^p \rightrightarrows \RR^p$ be a definable conservative field having a definable potential denoted by $f\colon \RR^p \to \RR$. Then $(f,D)$ has a $C^r$ variational stratification. In other words there exists a stratification $\{M_i\}_{i\in I}$ of $\R^p$ such that 
				$$P_{T_xM_x}D(x)=\{\grad f(x)\}$$
				whenever $M_x$ is the active strata.
						
				\label{prop:projImpliesChainRule}
\end{theorem}

\begin{proof} 
				We actually establish a slightly stronger result and prove that the result holds by replacing the underlying space $\R^p$ by a $C^r$ definable  finite dimensional manifold  $M$ with $D(x)\subset T_xM$, i.e., $D$ is a set valued  section of the tangent bundle. We follow a classical pattern of stratification theory, see e.g.,  \cite{dries1996geometric}, and we establish that the failure of the projection formula may only occur on a convenient small set. More precisely, we shall provide a stratification ${\cal S}=\{M_1,\ldots, M_q\}$ of $M$ for which each of the sets:
\begin{equation}
		\label{small}
    R_i=\{x\in M_i: \grad f(x)\neq P_{T_x M_i}D(x)\},\mbox{ (with } i=1,\ldots,q),
\end{equation}
has a dimension strictly lower than $M_i$.

Assuming we have such a stratification, let us see how  we would refine the stratification  in order to downsize further the set of ``bad points". For each $R_i$ we would consider the stratification $R^1_i,\ldots,R^r_i$ of $f_{|R_i}$ into smooth functions. For each $j=1,\ldots,r$ the couple $x\mapsto f|_{R^j_i}(x), x\mapsto P_{T_x{R^j_i}}D(x)$ would satisfy the assumption of the theorem but on a manifold with strictly lower dimension. So we could then pursue the process and conclude by exhaustion.

To obtain $\SSS$, we  consider a ``constant-rank" Whitney stratification of $f$, $\SSS=\{M_1,\ldots, M_q\}$, i.e., such that $f$ is smooth on each $M_i$ and has a constant rank, $0$ or $1$ in our case (see \cite{dries1996geometric}).  Take $M_i$ an arbitrary strata and $R_i$ as in \eqref{small}. We only need to prove that $R_i$ has a  dimension lower than  that of $M_i$.

For simplicity set $R_i=R$, $M_i=M$. We consider first the rank 0 case and deduce the other case afterward.

Assume that $\rank\, f=0$ on $M$, i.e., $f$ is constant. We want to prove that for almost all $x$ in $M$, $\max\{\|v\|: v\in P_{T_xM}D(x)\}=0$. We argue by contradiction and assume that  we have a ball of radius $\rho>0$ and center $\bar x$ on which the max is strictly greater than a positive real $m$. Consider the mapping  $G \colon x \rightrightarrows \arg\max\{\|v\|: v\in P_{T_xM}D(x)\}$ which is definable with nonempty compact values. Use the definable choice's theorem to obtain a definable single-valued selection $H$ of $G$, see \cite{Cos99}. $H$ is a (nonsmooth) vector field on $M$ that we may stratify in  a way compatible with $B_M(\bar x, \rho)$. The ball $B_M(\bar x, \rho)$ must contain a stratum of maximal dimension and hence there exists $\hat x \in B_M(\bar x,\rho)$, and $\epsilon\in (0, \rho)$ such that, $H(x)\in T_xM$ is smooth over $B_M(\hat x,\epsilon)\subset B_M(\bar x,\rho)$. We may thus consider the curve defined by
$$\dot \gamma(t)= H(\gamma(t)),\: \gamma(0)=\hat x.$$
For this curve, which is non stationary, one has almost everywhere
$$\frac{d}{d t}f(\gamma(t))=\max_{v\in D (\gamma(t))} \langle \dot \gamma, v\rangle=\max_{v\in P_{T_xM}D f(\gamma(t))}\|v \|^2\geq m^2>0,
$$
which is in contradiction with the fact that $f$ is constant. This concludes the null rank case.

Assume now that $\rank\, f=1$, so that $\grad f$ is nonzero all throughout $M$. Consider $\tilde{D} = D - \grad f$ which is definable, convex valued and has a closed graph. By linearity of the integral $\tilde{D}$ is conservative and has zero as a potential function over $M$. Indeed if $\gamma:[0,1]\to M$ is an arbitrary absolutely continuous curve, we have the set valued identity:
\begin{eqnarray*}
\int_0^1 \langle \tilde D\gamma(t),\dot \gamma(t)\rangle \dd t & = & \int  \langle D(\gamma(t)),\dot \gamma(t)\rangle \dd t - \int_0^1 \langle \grad f(\gamma(t)),\dot \gamma(t)\rangle \dd t\\
 & = & f(\gamma(1))-f(\gamma(0))-(f(\gamma(1))-f(\gamma(0)))\\
 & = & 0.
\end{eqnarray*}
 Since the null function has rank 0 on $M$, we deduce as above that $P_{T_xM}\tilde{D}(x)=\{0\}$ for almost all $x \in M$. Since $\tilde{D} = D - \grad f$ and $\grad f(x) \in T_xM$ for all $x \in M$, we deduce that $P_{T_xM}D(x)=\{\grad f(x)\}$ for almost all $x \in M$ which is what we needed to prove.\end{proof}

\begin{remark}[Definability]{\rm 
The fact that $f$ is definable does not imply that $D$ is definable, for instance setting $\chi(s)=[0,1]$ if $s$ is an integer and $0$ otherwise, $\chi$ is conservative for the zero function and it is not definable since $\mathbb{Z}$ is not definable in any o-minimal structure.  However if $f$ is definable in $\cal O$  so is its Clarke subdifferential, see \cite{bolte2007clarke}. 
On the other hand,  when $D$  definable in $\cal O$, its potential $f$ is not in general definable in $\cal O$. Whether it is definable in a larger o-minimal structure is an open question.}
\end{remark}

\begin{remark}[Alternative proof]{\rm 
Another  method for proving Theorem \ref{t:var}  relies on the repeated use of Theorem \ref{t:Rade}. We chose to avoid the use of strong analysis results, as Rademacher theorem, and pertain to standard self-contained definable arguments.}
\end{remark}

\subsection{Geometric and dynamical properties of definable  conservative fields}
This section describes some properties of definable conservative fields (with definable potential function). The ideas and proofs are direct generalizations of \cite{bolte2007clarke}.

\begin{theorem}[Nonsmooth Morse-Sard for $D$-critical values]
    Let  $D \colon \RR^p \rightrightarrows \RR^p$ be a  conservative field for $f:\R^p\mapsto\R$ and assume that $f,D$ are definable. Then the set of $D$-critical values $\{f(x),\, x \in \RR^p,\, 0 \in D(x)\}$ is finite.
    \label{th:morseSard}
\end{theorem}
\begin{proof}
    The proof is  as in \cite[Corollary 5]{bolte2007clarke} and follows from the variational stratification property, applying the definable Sard theorem to each strata. This ensures that the set of critical values has zero Lebesgue measure in $\RR$ and since it is definable, it is a finite set.
\end{proof}

The following is a generalization of the result of Kurdyka \cite{kurdyka1998gradients}
\begin{theorem}[A nonsmooth KL inequality for conservative fields]    \label{th:KLinequality}
    Let  $D \colon \RR^p \rightrightarrows \RR^p$ be a  conservative field for $f:\R^p\mapsto\R$ and assume that $f,D$ are definable. Then there exists $\rho > 0$, $\varphi  \colon [0, \rho) \to \RR_+$, definable strictly increasing, $C^1$ on $(0,\rho)$ with $\varphi (0) = 0$ and a continuous definable function $\chi \colon \RR_+ \to (0, + \infty)$, such that for all $x \in \RR^p$ with $0< |f(x)| \leq \chi(\|x\|)$ and $v \in D(x)$,
    \begin{align*}
        \|v\| \varphi '(|f(x)|) \geq 1.
    \end{align*}
\end{theorem}
\begin{proof}
This is deduced from the variational stratification property  as in Theorem 14 of  \cite{bolte2007clarke}. 
\end{proof}

The following convergence result is a consequence and calls for many questions regarding nonsmooth generalized gradient systems \cite{kurdyka2000proof}.
\begin{theorem}[Conservative  fields curves have finite length]
    Let  $D \colon \RR^p \rightrightarrows \RR^p$ be a definable conservative field and $x \colon \RR_+ \to \RR^p$  a solution of the differential inclusion
    \begin{align*}
        \dot{x}(t) \in - \mathrm{conv}(D(x(t))).
    \end{align*}
    Then if $x$ is bounded, $x$ has finite length: $\displaystyle \int_0^{+\infty} \|\dot{x}(t)\|\dd t < + \infty$ and in particular $x$ is a convergent trajectory.
    \label{th:finiteLength}
\end{theorem}
\begin{proof} Assume without loss of generality that $D$ has convex values, set $\|D(x)\| := \min_{v \in D(x)} \|v\|$ for any $x \in \RR^p$. Let $f$ be a potential function for $D$. We have
\begin{align*}
    \frac{d}{dt} f(x(t)) = - \|D(x(t))\|^2
\end{align*}
for almost all $t \geq 0$. We deduce that $t \mapsto f(x(t))$ has a limit, say $0$ (otherwise shift $f$ by a constant). The limit points $\omega$ of $x$ are entirely contained in a compact zone of $[f=0]$. Uniformize the nonsmooth KL inequality on a tubular neighborhood, say $Z$, of this zone (see \cite[Lemma 6]{bolte2014proximal}), and finally assume that $x(t)\in Z$ for some $t\geq t_1$. On $Z$, set $\tilde{D}(x) =  \varphi '(f(x))D(x)$ and observe that $\tilde{D}$ is a conservative field for $\varphi  \circ f$, hence for almost all $t \geq t_1$
\begin{align*}
    \frac{d}{dt} \varphi  \circ f(x(t)) &= \left\langle \dot{x}(t), \varphi '(f(x)) D(x) \right\rangle \\
    &= - \|\dot{x}(t)\|^2  \varphi '(f(x)) \leq -\|\dot{x}(t)\| \varphi '(f(x)) \|D(x)\| \leq -\|\dot{x}(t)\|.
\end{align*}
Since $\varphi (f(x(t))$ tends to $0$, we obtain that $\displaystyle\int_0^{+\infty} \|\dot{x}\| \leq \varphi (f(x(0)))$.
\end{proof}

\section{Automatic differentiation}
Automatic differentiation emerged in the seventies as a computational framework which allows to compute efficiently gradients of multivariate functions expressed through smooth elementary functions. When the function formula involves nonsmooth elementary functions the automatic differentiation approach fails to provide gradients. This issue is largely studied in \cite[chapter 14]{griewank2008evaluating} which discusses connections with Clarke generalized derivatives using notions such as ``piecewise analyticity" or ``stable domain". 
Let us mention \cite{griewank2013stable} which  developed piecewise linear approximation for functions which can be expressed using absolute value, min or max operators. This approach led to successful algorithmic developments \cite{griewank2016lipschitz} but may suffer from a high computational complexity and a lack of versatility (the Euclidean norm cannot be dealt with within this framework). Another attempt using the same model of branching programs was described in \cite{kakade2018provably} where a qua\-li\-fi\-ca\-tion assumption is used to compute Clarke generalized derivatives automatically\footnote{From a practical point of view, qualification is hard to enforce or even check.}.

We provide now a simple and flexible theoretical model for automatic differentiation through conservative fields.

\subsection{A functional framework: ``closed formula functions" }
Automatic differentiation deals essentially with composed functions whose components are ``simple functions", that is functions coming as ``closed formulas". It presumes the existence of a chain rule and aggregates the basic derivation operations according to this principle. We refer to \cite{griewank2008evaluating} for a detailed account. The purpose of this section is to demonstrate that our nonsmooth differentiation model is perfectly fit to deal with this approach in the nonsmooth case.

The function $f$ we consider now is accessible through a recursive algorithm which materializes an evaluation process built on a directed graph. This graph\footnote{Which we shall not define formally since it is not essential to our purpose.} is modelled by  a discrete map called $\parents$  and a collection of known ``elementary functions" $g_k$ defined through:

\begin{itemize}
				\item[a)] $q \in \NN$, $q > p$
				\item[b)] $\parents$ maps the set $\left\{ p+1,\ldots,q \right\}$ into the set of tuples of the form $(i_1,\ldots,i_m)$ where $m \in \NN$ and $i_1, \ldots, i_m$ range over $\{1,\ldots,q-1\}$ without repetition. It has the property that for any $k \in \left\{ p+1,\ldots, q \right\}$, $\parents (k)$ is a tuple without repetition over the indices $\left\{ 1,\ldots,k-1 \right\}$. 
				\item[c)] $\left( g_i \right)_{i=p+1}^q$ such that for any $i = p+1, \ldots, q$, $g_i\colon \RR^{|\parents (i)|} \to \RR$.
\end{itemize}

\begin{algorithm}[ht]
  \caption{Definition program of $f\colon \RR^p \to \RR$}
	\label{alg:algof}

  \begin{algorithmic}[1]

  \item[\textbf{Input:}] $x=(x_1, \ldots x_p)$
    \FOR{$k=p+1,p+2,\ldots q$}
    \STATE Set:
    \[
    x_k = g_{k} (x_{\parents (k)})
    \]
		where $x_{\parents (k)} = \left( x_i \right)_{i \in \parents (k)}$.
    \ENDFOR
    \item[\textbf{Return:}] $x_q=:f(x)$.
\end{algorithmic}
\end{algorithm}

This defines the function $f$ through an operational evaluation  program. 
 \begin{example}\label{ex:exampleAutoDiff}{\rm
				The idea behind automatic differentiation is that the original function is given through a closed formula, which is then interpreted as a composed function in order to make its differentiation (or ``subdifferentiation") amenable to simple chain rule computations. For instance for $f(x)=(x_1x_2+\tan x_2)(|x_1|+x_1x_2x_3)$, we may choose 
				\begin{align*}
				&x_4=g_4(x_1,x_2)=x_1x_2, \: x_5=g_5(x_2)=\tan x_2, \: x_6=g_6(x_1)=| x_1|,\\
				&x_7=g_7(x_3,x_4)=x_3x_4, \\
				& x_8=g_8(x_4,x_5,x_6,x_7)=(x_4+x_5)(x_7+x_6)
				\end{align*}
				where the $\parents$ function is in evidence $\parents(4)=\{1,2\}$, $\parents(5)=\{2\}$, $\parents(6)=\{1\}$, $\parents(7)=\{3,4\}$, $\parents(8)=\{4,5,6,7\}$. Observe that the derivatives or subdifferentials of $g_4,\ldots,g_8$ are known in closed form. Concerning $g_6=|\cdot|$ one has $\partialc g_6(0)=[-1,1]$. Thus in practice we need to choose a specific element in that set, as $0$, and perform the computation with this choice (see below the forward or backward differentiation modes).}
\end{example}

\subsection{Forward and backward nonsmooth automatic differentiation} In order to compute a conservative field for $f$, we need in addition the following: 
\begin{itemize}
				\item[d)] For any $i = p+1, \ldots, q$, one is given a set valued map $D_i\colon \RR^{|\parents (i)|} \rightrightarrows \RR^{|\parents (i)|}$ which is conservative for $g_i$.
\end{itemize}
Note that the above implies that the $g_i$ are locally Lipschitz continuous.

For example, $D_i$ could be the Clarke subgradient of $g_i$ if $g_i$ is definable (a mere  definable selection in the Clarke would also work). For instance in Example \ref{ex:exampleAutoDiff}, one may set $D_6(0)=\{0\}$ or $D_6(0)=[0,1]$.
Given $\left( x_i \right)_{i=1}^q$ as computed in Algorithm \ref{alg:algof}, an algorithm to compute a conservative field of $f$ is described in Algorithm \ref{alg:autodiff0}. This is a direct implementation of the chain rule as described in Lemma \ref{lem:compositionSubgradient}. This ensures that the output of Algorithm \ref{alg:autodiff0} is a conservative field for the function $f$ described in Algorithm~\ref{alg:algof}. Furthermore, the reverse mode of automatic differentiation described in Algorithm \ref{alg:autodiff} computes  essentially  the same quantity but with a lower memory and time footprint.

\begin{algorithm}[t]
  \caption{Forward mode of  automatic differentiation for $f$}
	\label{alg:autodiff0}
   \begin{algorithmic}[1]
	 \item[\textbf{Input:}] variables $(x_1, \ldots x_q)$; $d_i =(d_{ij})_{j=1}^{|\parents (i)|}\in D_i(x_{\parents (i)})$, $i = p+1 \ldots q$
    \STATE Initialize: $
    \frac{\partial x_k}{\partial x} = e_k$, $k = 1,\ldots, p$.
    \FOR{$k= p+1, \ldots q$}
    \STATE Compute:
    \[
    \frac{\partial x_k}{\partial x} = \sum_{j \in \parents (k)} \frac{\partial x_j }{\partial x} d_{kj}
    \]
    \mbox{where $x=(x_1,\ldots,x_p).$}

    \ENDFOR
    \item[\textbf{Return:}]
$\frac{\partial x_q}{\partial x}$.
  \end{algorithmic}
\end{algorithm}

\begin{algorithm}[t]
  \caption{Reverse Mode of automatic differentiation for $f$}
	\label{alg:autodiff}
   \begin{algorithmic}[1]
   \item[\textbf{Input:}] variables $(x_1, \ldots x_q)$; a
    the map $\{\parents (t)\}_{t\in \{1,\ldots q\}}$;
    associated derivatives $d_i =(d_{ij})_{j=1}^{|\parents (k)|}\in D_i(x_\parents (i))$, $i = p+1 \ldots q$
    \STATE Initialize: $ v = (0,0,\ldots, 0,1) \in \RR^q$
    \FOR{$t= q, \ldots p+1$} 
    \FOR{$j \in \parents (t)$}
        \STATE Update coordinate $j$ of $v$:
    \[
    v[j] \mathrel{:=}  v[j] + v[t] d_{tj}\quad \text{ i.e., update } v[j] \text{ to }  v[j] + v[t] d_{tj}
    \]
    \ENDFOR
    \ENDFOR
    \item[\textbf{Return:}]
$\left(v[1], v[2], \ldots, v[p]\right) $.
  \end{algorithmic}
\end{algorithm}

				Let $f$ be gi\-ven as in Algorithm \ref{alg:algof} and $x \in \RR^q$ be given the trace of all intermediate values at the end of Algorithm \ref{alg:algof}. Set 
				\begin{align*}
								D = \{v \in \RR^p;\; \text{ output of Algorithm \ref{alg:autodiff0}}\}
				\end{align*}
				for all possible choices of $d_k \in D_k(x_{\parents(k)})$, $k=p+1, \ldots, q$. The first $p$ coordinates of $x$ are the arguments of $f$, repeating the same construction for any value of these coordinates defines a set valued field on $\RR^p$ with values in subsets of $\RR^p$. This field is called the {\em forward automatic differentiation field} of Algorithm \ref{alg:algof}. The {\em backward automatic differentiation field of Algorithm~\ref{alg:algof}} is defined similarly using Algorithm~\ref{alg:autodiff}.

\begin{theorem}[Forward and backward ``autodiff" are conservative fields]
				Let $f$ be  gi\-ven through Algorithm \ref{alg:algof}. Then 
				the forward and backward automatic differentiation fields are conservative for $f$.
  \label{l:autodiff}   
\end{theorem}
\begin{proof}
    We substitute the functions $\left(g_k\right)_{k=p+1}^q$ by functions $\left(G_k\right)_{k=p+1}^q$, such that for each $k = p+1,\ldots,q$
    \begin{align*}
        G_k \colon \RR^q &\mapsto\RR^q \\
        x &\mapsto x + e_k(g_k(x_{\parents (k)}) - x_k),
    \end{align*}
    where $e_k$ is the $k$-th element of the canonical basis. Similarly, for all $k \in p+1,\ldots,q$, we still denote by $D_k \colon \RR^q \rightrightarrows \RR^q$ the conservative field of $g_k$ seen as a function of $x_1,\ldots, x_q$ (simply add zeros to coordinates which do not correspond to parents of $k$). Then $f$ as computed in Algorithm \ref{alg:algof}, is equivalently given by
    \begin{align*}
        f(x) = \left [\left(G_q \circ G_{q-1} \circ \ldots \circ G_{p+1} (x_1, \ldots, x_p, 0,\ldots,0)\right)\right]_q =  \left[G_q \circ G_{q-1} \circ \ldots \circ G_p (x)\right]_q
    \end{align*}
    where $G_p$ maps the first $p$ coordinates $(x_k)_{k=1}^p$ to the vector $((x_i)_{i=1}^p, (0)_{i=p+1}^q) \in \RR^p$ and indexation $[\cdot]_q$ denotes the $q$-th coordinate of a $q$ vector.
    For each $k = p+1 \ldots, q$, $x \in \RR^q$, the  following ``componentwise derivative" of $G_k$ in a matrix form
    \begin{align}
        L_k \colon x \mapsto \left\{I - e_k e_k^T + e_k d^T, \quad d \in D_k(x)\right\},
        \label{eq:conservativeMatrixK}
    \end{align}
   is a conservative mapping for $G_k$ by Lemma \ref{lem:DIjacobian}. For each $k = p+1 \ldots, q$, we choose one such matrix according to a fixed input of Algorithm \ref{alg:autodiff0}, $d_k \in D_k(x_1,\ldots, x_q)$,
    \begin{align}
        J_{k} =  I - e_k e_k^T + e_k d_k^T \in L_k(x_1,\ldots, x_q)
        \label{eq:conservativeMatrixK2}
    \end{align}
    For each $k = p+1,\ldots q$, denote by $M_k$ the matrix defined by blocks as follows
    \begin{align*}
        M_k = 
        \begin{pmatrix}
             I_p  \\
              \left(\frac{\partial x_{p+1}}{\partial x}\right)^T\\
              \vdots\\
              \left(\frac{\partial x_k}{\partial x}\right)^T\\
              0
        \end{pmatrix} \in \RR^{q \times p}
    \end{align*}
    where $\frac{\partial x_k}{\partial x}$ is computed by Algorithm \ref{alg:autodiff0}.
    Denote also by $J_{p} \in \RR^{q \times p}$ the diagonal matrix which diagonal elements are $1$ and the remainders are $0$, the Jacobian of $G_p$. One can see that
    \begin{align*}
        M_k = J_{k} \times J_{k-1} \times \ldots \times J_{p+1} \times J_{p}
    \end{align*}
		for all $k = p+1, \ldots q$, where $\times$ denotes the usual matrix product. This is easily seen for $M_{p+1}$ as Algorithm \ref{alg:autodiff0} computes $\frac{\partial x_1}{\partial x_{1,\ldots,p}} = d_1$. The rest is a simple recursion. In the end Algorithm \ref{alg:autodiff0} computes
    \begin{align*}
        e_q^T M_q &= e_q^T \times J_{q} \times J_{q-1} \times \ldots \times J_{p+1} \times J_{p} \\
        &\in e_q^T\times L_{q} \times L_{q-1} \times \ldots \times L_{p+1} \times J_{p}
    \end{align*}
    Combining  Lemmas \ref{lem:projRowJacobian} and \ref{lem:compositionJacobian}, the right hand side is a conservative field for $f$. Actually it can be seen from equations \eqref{eq:conservativeMatrixK} and \eqref{eq:conservativeMatrixK2} that the right hand side consists precisely of all possible outputs of Algorithm \ref{alg:autodiff0} for all possible choices of $d_k$, $k = p+1, \ldots, q$. This proves the claim for the forward automatic differentiation field obtained by Algorithm~\ref{alg:autodiff0}. 
    
    Regarding Algorithm \ref{alg:autodiff}, we will show that it computes the same quantity reversing the order of the products. For all $t = q,\ldots, p+1$, let $v_t \in \RR^q$ be the vector $v$ obtained after step $t$ of the ``for loop" of Algorithm \ref{alg:autodiff}. We have $v_q = e_q + d_q = (I + d_q e_q^T) e_q$. An induction shows that for all $t = q, \ldots, p+1$  
    \begin{align*}
        v_t = (I + d_t e_t^T) \ldots (I + d_q e_q^T) e_q.
    \end{align*}
    Using the same notation as in equation \eqref{eq:conservativeMatrixK2}, set for $t = q,\ldots p+1$
    \begin{align*}
        w_t = J_{t}^T \times \ldots \times J_{q}^T \times e_q
    \end{align*}
    It is easy to see that $w_q$ and $v_q$ agree on the first $q-1$ coordinates. By recursion, for $t = q,\ldots p+1$, $w_t$ and $v_t$ agree on the first $t-1$ coordinates (recall that $d_t$ is supported on the first $t-1$ coordinates). We deduce that $w_{p+1}$ and $v_{p+1}$ agree on the first $p$ coordinates so that the output of Algorithm \ref{alg:autodiff} is 
    \begin{align*}
        J_{G_p}^T v_{p+1} =  J_{G_p}^T w_{p+1} =  J_{G_p}^T \times J_{G_{p-1}}^T \times \ldots \times J_{G_q}^T \times e^q = M_q^T e_q
    \end{align*}
    which is the same quantity as the one computed by Algorithm \ref{alg:autodiff0}. The claim follows for the backward automatic differentiation field.
\end{proof}

\begin{remark}{\rm 
                Automatic differentiation is not necessarily convex valued. Consider the function 
                \begin{align*}
                    f \colon (x,y) \mapsto |\max(x,y)|.
                \end{align*}
                Both the max and absolute value functions are convex so that  their respective convex subgradients are  conservative fields.
                Applying the chain rule in Lemma \ref{lem:compositionJacobian} at $x = y = 0$ we obtain a conservative field for $f$ evaluated at zero of the form 
                \begin{align*}
                    D=\left\{t v;\; t \in [-1,1], \, v \in \Delta \right\}
                \end{align*}
                where $\Delta$ is the one dimensional simplex in $\RR^2$. The set $D$ is not convex.
                
               }
\end{remark}

The following corollary is a direct consequence of Theorems \ref{th:stratification} and \ref{l:autodiff}. Note that a result close to equation \eqref{Df=grad} was already guessed in \cite[Proposition 14.2]{griewank2008evaluating}.
\begin{corollary}[Automatic differentiation returns the gradient a.e.] Assume that all the $g_k$ defining $f$ and their conservative fields $D_k$, are definable, and denote by $D_f$ a resulting automatic differentiation field (either forward or backward). Then $f$ is differentiable almost everywhere with
                \begin{equation}\label{Df=grad}
                   D_f=\{\nabla f\} 
                \end{equation}
                 on the complement of finitely many smooth manifolds with dimension at most $p-1$. Furthermore, for any $v,w$ in $\R^p$,
                 \begin{align}\label{intDfSegment}
                   f(w)-f(v)&=\int_0^1\left\langle D_{f}((1-t)v+tw),w-v\right\rangle\dd t.
                \end{align}
\end{corollary}
\begin{proof} From Theorem \ref{l:autodiff}, $D_f$ is a conservative field for $f$. Basic closedness properties of definable objects ensure that both $f$ and $D_f$ are definable so that Theorem \ref{prop:projImpliesChainRule} ensures the existence of a variational stratification (see Definition \ref{def:projectionFormula}). The fact that $f$ is differentiable almost everywhere is a basic result of tame geometry. To obtain \eqref{Df=grad}, use the stratification provided in Theorem \ref{prop:projImpliesChainRule}, and consider the dense open set given by the union of the finite number of strata of maximal dimensions. The integration formula is the application of Definition \ref{def:conservativeMapForF} along  segments. \end{proof}
                
\begin{remark}[The limitations of the smooth chain rule] \label{smoothchainrule}     {\rm  (a) It is surprising to use Theo\-rem \ref{prop:projImpliesChainRule} which is non trivial to obtain \eqref{Df=grad}. It is instead tempting to simply use the expression of $f$ provided in Theorem  \ref{l:autodiff}: 
                $$f(x)=e_q^TG_p\circ \ldots\circ G_q (x)$$ and to differentiate it ``almost everywhere" to obtain 
                $$f'(x)=e_q^TG'_p\left(G_{p-1}(\ldots G_q(x))\ldots)\right)\circ \ldots\circ G'_q(x),$$
                which would give the desired result. 
Unfortunately this expression has no obvious meaning, since for instance, the image of $G_q$ may be entirely contained in the points of non-differentiability of $G_{q-1}$, so that $G'_{q-1}(G_q(x))$ has no meaning. This result is illustrated further in the deep learning section through an experimental example.\\
(b) Observe as well that we do not know  whether such a result could hold without definability assumptions.
 }
   \end{remark}

\section{Algorithmic consequences and deep learning}\label{s:sto}

What follows is in the line of many works on decomposition methods \cite{bottou2008tradeoffs}, in particular those involving nonconvex problems, see e.g.,  \cite{moulines2011nonasymptotic,bottou2018optimization,davis2018stochastic,majewski2018analysis,chizat}. Our study uses  connections with dynamical systems, see e.g.,   \cite{ljung1977analysis,kushner20003stochastic,benaim2005stochastic,borkar2009stochastic,bianchi2019constant} in order to take advantage of the ``curve friendly" nature of conservative fields and automatic differentiation. Using our for\-ma\-lism,  we gather ideas  from \cite{davis2018stochastic,castera2019inertial,adil}, and use the Bena\"im-Hofbauer-Sorin approach \cite{benaim2005stochastic}, to obtain almost sure subsequential convergence to steady states that are carefully defined. To our knowledge, this provides the first proof for the subsequential convergence of the stochastic gradient descent  with mini-batches in deep learning when the actual  backpropagation model is used instead of the subgradient one's, which is the case in almost all applications involving nonsmooth objects.  As outlined in a conclusion, many more algorithms could be considered along this perspective.

All sets and functions we consider in this section are definable in the same o-minimal structure.
\subsection{Mini-batch stochastic approximation for finite nonsmooth nonconvex sums aka ``nonsmooth nonconvex SGD"}
We consider the following loss function on $\RR^p$
\begin{align}\label{Jmodel}
    \J \colon w \mapsto \frac{1}{n} \sum_{i=1}^n f_i(w)
\end{align}
where each $f_i \colon \RR^p \to \RR$ is definable and locally Lipschitz continuous. We assume that for each $i = 1 \ldots n$, $D_i \colon \RR^p \rightrightarrows \RR^p$ is a definable  conservative field for $f_i$, for example the ones provided by automatic differentiation. We consider the following recursive process, given a sequence of nonempty mini-batches subsets of $\{1,\ldots,n\}$, $\left( B_k \right)_{k \in \NN}$, taken independently, uniformly at random, $\left( \alpha_k \right)_{k \in \NN}$ a deterministic sequence of positive step sizes, and $w_0 \in \RR^p$, iterate 
$$\mbox{Mini-batch SGD with back-prop: }\qquad\left\{
\begin{aligned}
& w_{k+1}  = w_k - \alpha_k d_k\\ 
& d_k \in \frac{1}{|B_k|}\sum_{i \in B_k} D_i(w_k). \label{eq:miniBatch}
\end{aligned}
\right.\qquad\qquad\qquad$$
We set 
\begin{align*}
        D_\J \colon w \rightrightarrows \frac{1}{n} \conv\left(\sum_{i=1}^n D_i(w) \right)   
\end{align*}
and $\crit_J = \{w \in \RR^p,\, 0 \in D_\J(w)\}$, the set of $D_\J$-critical points.
Combining our results with the approach of \cite{benaim2005stochastic}, we obtain the following asymptotic characterization.
\begin{theorem}[Convergence of mini-batch SGD]\label{t:sgd}       
				Let $\alpha_k = o(1/\log(k))$, denote by $\BB$ the event defined by $\sup_{k} \|w_k\| < + \infty$ and assume that $\BB$ occurs with a positive probability. 
				
				Setting, $\bar{w} \subset \RR^p$, the set of accumulation points of $(w_k)_{k\in \NN}$, we have, almost surely on $\BB$, $\emptyset \neq \bar{w} \subset \crit_\J$ and $\J$ is constant on $\bar{w}$.
\end{theorem}
\begin{proof}
    $D_\J$ is a conservative field for $\J$. 
    Hence $\J$ is a Lyapunov function for $\crit_\J$ and the differential inclusion
    \begin{align*}
        \dot{w} \in - D_\J(w),
    \end{align*}
    which admits solutions according to \cite[Chapter 2, Theorem 3]{aubin1984differential}. 
    We have by uniform randomness
    \begin{align*}
        \EE_{B_k} \left[\frac{1}{|B_k|}\sum_{i \in B_k} D_i(w_k) \right] = \frac{1}{n} \sum_{i=1}^n D_i(w_k) \subset D_\J(w_k).
    \end{align*}
		In the event of $\BB$, the sequence remains bounded so that, almost surely on $\BB$, the following is also bounded
    \begin{align*}
      &  \sup_k \|d_k - v\| \\
       & \mathrm{s.t.} \quad v \in \EE_{B_k} \left[\frac{1}{|B_k|}\sum_{i \in B_k} D_i(w_k) \right]
    \end{align*}
    Theorem \ref{th:stratification} implies that $\J(\crit_\J)$ is finite, and hence has empty interior. 
		The result follows by combining Theorem 3.6, Remark 1.5(ii) and Proposition 3.27 of \cite{benaim2005stochastic}, see also \cite[Proposition 4.4]{benaim1999dynamics} for discussion on the step size and \cite[Section 2.2]{borkar2009stochastic} for a discussion on sequence boundedness.
\end{proof}

\begin{remark}[Convergence]{\rm  We conjecture that, beyond subsequential convergence, ite\-ra\-tes should converge in the case of definable potentials.}
\end{remark}

\subsection{Deep Neural Networks and nonsmooth backpropagation}
	
We pertain to feed forward neural networks even though much more general cases are adapted to our auto-differentiation setting and to our definability assumptions.  

Let us consider two finite dimensional real vector spaces spaces $\XX,\YY$. The space $\XX$ models input objects of interest (images, economical data, sentences, texts) while $\YY$ is an output space formed by properties of interest for the objects under consideration. The points $y$ in $\YY$ are often called labels. The goal of deep learning is to label automatically objects in $\XX$ by ``learning the labelling principles" from a large dataset of known paired vectors $(x_i,y_i)_{i=1,\ldots,n}$. Given $x$ in $\XX$, we thus wish to discover its label $y$. This is done by designing a predictor function whose parameters are organized in $L$ layers, each of which is represented by an affine function $A_j \colon \RR^{p_j} \to \RR^{p_{j+1}}$  for values $p_j \in \NN$, $j = 1,\ldots, L$. Our predictor function has then the compositional form
\begin{align}
    \XX \ni x\to \sigma_L(A_L(\sigma_{L-1}(A_{L-1}(\ldots\sigma_2(A_2(\sigma_1 (A_1 (x)))) \ldots))))\in \YY
    \label{eq:predictorDeepLearning}
\end{align}
	where for each $j$, the function $\sigma_j \colon \RR^{p_j} \to \RR^{p_j} $,  is locally Lipschitz continuous. These functions are called {\em activation functions} and are usually univariate functions applied coordinatewise. Very often one simply takes a single activation function $\sigma \colon \RR \to \RR$ and apply it to coordinates of each layer. 
Classical choices for $\sigma$ include:
\begin{enumerate}
				\item identity: $t \mapsto t$,
				\item sigmoid: $t \mapsto \frac{1}{1 + e^{-t}}$,
				\item hyperbolic tangent: $t \mapsto \mathrm{tanh}(t)$,
				\item softplus: $t \mapsto \log(1 + \exp(t))$,
				\item ReLU: $t \mapsto \max\{0,t\}$, aka positive part,
				\item ``Leaky-ReLU'': $t \mapsto \max\{0,t\} + \alpha \min\{t, 0\}$, $\alpha > 0$, parameter.
				\item piecewise polynomial activations.
\end{enumerate}
Examples 1, 5, 6, 7 are semialgebraic, the others are  definable in the same o-minimal structure ($\RR$-exp definable sets). Among these examples, the ReLU activation function \cite{glorot2011deep} played a crucial role in the development of deep learning architectures as it was found to be efficient in reducing ``vanishing gradient'' issues (those being related to the flatness of the commonly used sigmoid). This activation function is still widely used nowadays and constitutes one of the motivations for studying deeper automatic differentiation oracles applied to nonsmooth functions.

In order to lighten the notations, the weights of all the $A_i$ in \eqref{eq:predictorDeepLearning} are concatenated into a global weight vector $w$ in $\R^p$, so we may simply write the parametrized predictor with parameter $w$, 
$$g(w,x):=\sigma_L(A_L(\sigma_{L-1}(\ldots\sigma_1 (A_1(x)))))).$$ 
Learning a predictor function is finding an adequate collection of weights $w$. To do so one trains the neural networks by minimizing a loss of the form: 
	\begin{equation}
	 \J(w)=\frac{1}{n}\sum_{i=1}^n l(g(w,x_i),y_i)   
	\end{equation}
	where $l$ is some elementary loss function, typical choices include the square loss $l(a,b)=\frac12\|a-b\|^2$, $(a,b) \in \RR^2$ for regression or binary cross entropy for classification: $l(a,b) = b \log(a) + (1-b) \log(1-a)$, where $a \in (0,1)$, $b \in \{0,1\}$. 
	In view of matching the abstract model \eqref{Jmodel}, set $f_i(x)=l(g(w,x_i),y_i)$ for all $i$. It is obvious to see that:
\begin{lemma}[Deep Learning loss in algorithmic form]
				Given $\sigma_1,\ldots,\sigma_L$ and $l$, each term $f_i$  of the deep learning loss $\J$ has a representation as in Algorithm 1.
\end{lemma}

Let us now fix $\sigma_1,\ldots,\sigma_L$ and $l$. Choose a conservative map\footnote{If a unique $\sigma\colon \RR \to \RR$ is applied to each coordinate of each layer, this amounts to consider a conservative field for $\sigma$, for example its Clarke subgradient.} $D_{i}$ for each $\sigma_{i}$, $i = 1 \ldots, L$, and $D_l$ for $l$. An index $i$ being fixed, the backpropagation algorithm applied to $f_i$ is exactly backward auto-differentiation over $f_i$ based on the data of $\{D_{i}\}_{i=1 \ldots L}$ and $D_l$. We denote by $BP_{f_i}$ the output mapping. We have:
	\begin{corollary}[Backpropagation defines a conservative field]\label{l:BP}
					With the above conventions, assume that $l$ and $\sigma_1,\ldots,\sigma_L$ as well as the corresponding conservative maps are definable in the same o-minimal structure, then the mapping $BP_{f_i}$ is a conservative field. As a consequence 
	  $$BP_{f_i}=\nabla f_i $$
	  save on a finite union of manifolds of dimension at most $d-1$.
	  
	  As a consequence, setting \begin{equation}\label{JBP}
BP_{\J}=\frac{1}{n}\sum_{i=1}^n BP_{f_i}\end{equation}
we obtain a conservative field, and thus
\begin{align}
& BP_{\J}=\nabla \J \mbox{ a. e.}\\
& \J(w)-\J(v)=\int_0^1\left\langle BP_{\J}((1-t)v+tw),w-v\right\rangle\dd t,
\end{align}
for any $v,w$ in $\R^p$. 
	\end{corollary}

	\begin{remark}[Backpropagation and differentiability a.e.]{\em (a) 
	The backpropagation algorithm was popularized in the context of neural networks by \cite{rumelhart1986learning} and is at the heart of virtually the totality of numerical algorithms for training deep learning architectures \cite{lecun2015deep,abadi2016tensorflow,paszke2017workshop}. 
Most importantly, and this was the main motivation for our work, the backpropagation algorithm is used even for network built with non differentiable activation functions  one of the most well known example being ReLU \cite{glorot2011deep}. Using such nondifferentiable functions completely destroys the interpretation of backpropagation algorithm as computing a gradient or even a subgradient. Our results say that, although not computing any kind of known subdifferential, the nonsmooth backpropagation algorithm computes elements of a conservative field. As a consequence, it satisfies the operational chain rule given in Lemma \ref{lem:chainRuleAC}. Note also that virtually all deep network architectures used in applications are actually definable, see e.g., \cite{davis2018stochastic}.\\ 
(b) Despite our efforts we do not see any means  to obtain Corollary~\ref{l:BP} easily. In the ``compositional course of loss differentiation" (recall Algorithm 1 and 3), one can indeed  get trapped in ``nondifferentiability zones" and thus speaking of the derivative of the active layers at this point has no meaning. Thus the smooth chain rule is of no use (see Remark \ref{smoothchainrule}) and the nonsmooth chain rules, for limiting or Clarke subdifferential are simply false in general, see for example \cite{kakade2018provably}. 

To illustrate the fact that nonsmooth zones can be significantly activated during the training phase, we present now a numerical experiment. Let us consider a very simple feed forward architecture composed of $L$ layers of fixed size $p$. Each layer is computed from the previous layer by application of a linear map, from $\RR^p$ to $\RR^p$, composed with a coordinatewise form of ReLU. The input layer is the first element of the canonical basis and we sample the weights matrices with iid uniform entries in $[-1,1]$. We repeat this sampling many times and estimate empirically the probability of computing $\mathrm{ReLU}(0)$ during forward propagation of the network (this would require to use the derivative of ReLU at $0$ during backpropagation).

The results are depicted in Figure \ref{fig:probaZero}. It appears very clearly that for some architectures, with nonvanishing probability, we sample weight matrices resulting in the computation of $\mathrm{ReLU}(0)$. This means that, although the output of the network is piecewise polynomial as a function of weight matrices, and hence almost everywhere differentiable, {\em we still need to evaluate intermediate functions at points where they are not differentiable with non zero probability}. Hence, as we already mentioned, one cannot assert  that the fact that the output is differentiable almost everywhere implies that the classical chain rule of differentiation applies almost everywhere; this assertion is just false. These empirical results also confirm the interest of working with piecewise smooth objects through stratification techniques.

\begin{figure}
    \centering
    \includegraphics[width=.8\textwidth]{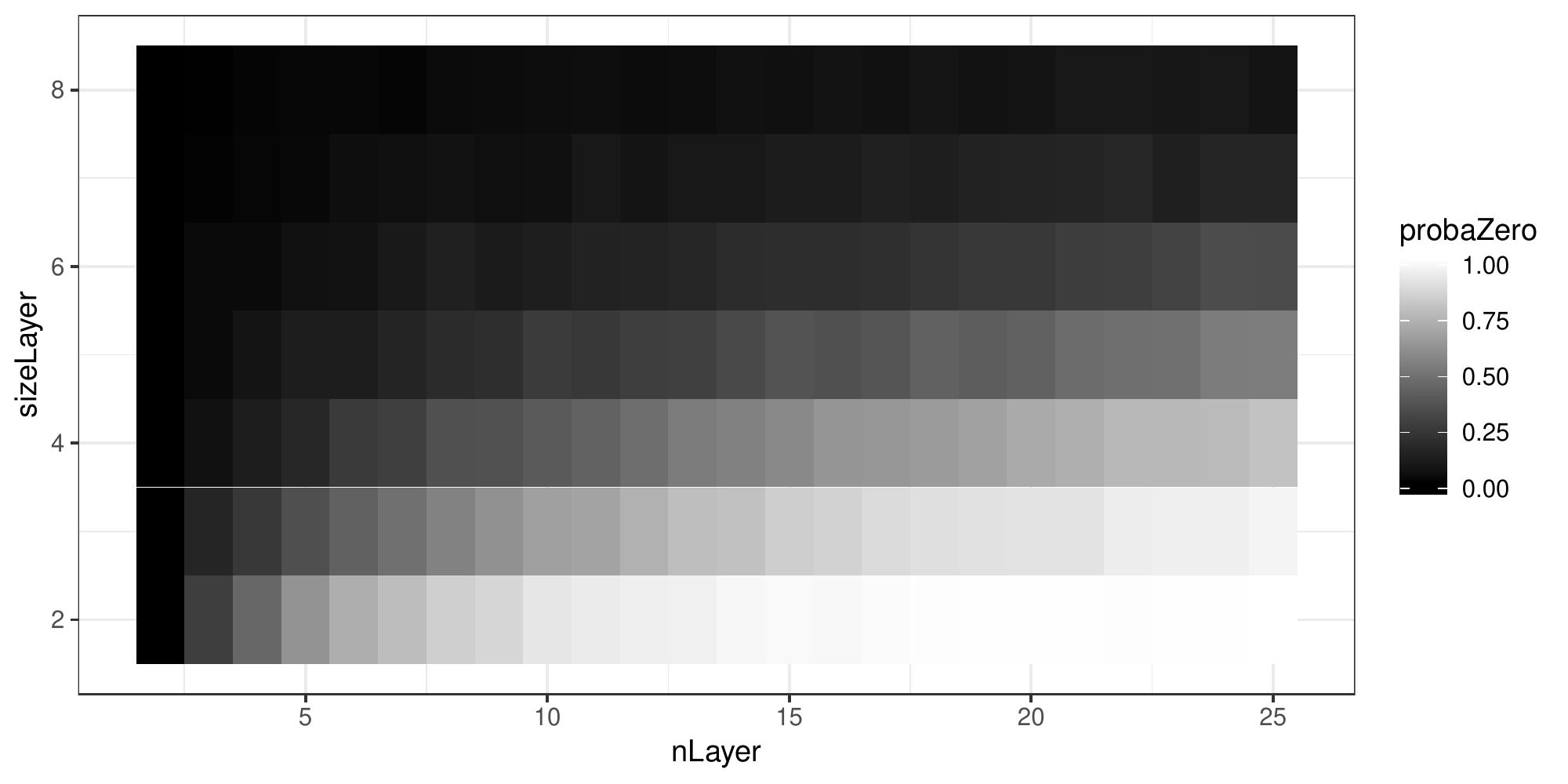}
    \caption{\small Estimation of the probability of applying ReLU to $0$, as a function of the size and number of layers in a feedfoward network. The input is set to the first element of the canonical basis and we then propagate application ReLU layers with linear functions. The weights of the linear term are sampled uniformly at random between -1 and 1.}
    \label{fig:probaZero}
\end{figure}

}\end{remark}

\subsection{Training nonsmooth neural networks with nonsmooth SGD}

To our knowledge the following result is the first genuine analysis of {\em nonsmooth} SGD for deep learning taking into account the real nature of backpropagation and the use of mini-batches. Note that the steady states below, $BP_\J$-critical points (see \eqref{JBP}), are the {\em actual steady states} of the corresponding dynamics, which is generally ignored in the literature. For simplicity of reading, we consider the special case of ReLU networks with squared loss.

\begin{corollary}[Convergence of SGD for Deep Learning] Consider a feed for\-ward neu\-ral net\-work with mean squared error and {\rm ReLU} activation function. Then the bounded se\-quences generated by the mini-batch SGD algorithm  using the backpropagation oracle approach the $BP_\J$-critical set of the loss function with probability one.  
\end{corollary}
This is a direct consequence of Theorem \ref{t:sgd} since the squared norm and $\mathrm{ReLU}$ are semialgebraic. The same result holds with any functions $\sigma_1, \ldots, \sigma_L$ and $l$ definable in the same structure. As mentioned previously more complex architectures are accessible since our results rely only on abstract automatic differentiation and definability.

\section{Conclusion} We introduced new tools for {\em nonsmooth} nonconvex problems, based on the idea that the choice of a fixed notion of subdifferential right from the start can be extremely limiting in terms of analysis and even of representation (e.g., automatic differentiation).

Our approach eventually consists in the following protocol. 

Consider an optimization problem involving an automatic differentiation oracle --we focused on the example of deep learning, but other application fields are possible (numerical simulations, optimal control solvers or partial differential equations \cite{mohammadi2010applied,corliss2002automatic}).

\begin{itemize}

\item[--] {\bf ``Choose your optimization method and then choose your subdifferential".} Evaluate precisely your decomposition requirements, in terms of sum or product, e.g., mini-batches for SGD.
Infer from the decomposition method and the use of nonsmooth automatic differentiation a conservative field  matched to the considered  algorithm, e.g., coming back to SGD,  set $\displaystyle D_f=\frac1n\sum_{i=1}^n D_{f_i}$.

\item[--]{\bf ``Verify definability or tameness assumption".} Check that the various objects are definable in some common adequate structure. The problems we met are covered by one of the following, by order of frequency:  semialgebraicity, global subanalyticity or log-exp structures. 

\item[--] {\bf ``Identify Lyapunov/dissipative properties".} Use a Lyapunov approach, e.g.,  \`a la Bena\"im-Hofbauer-Sorin,  to conclude that the algorithm under consideration has dissipative properties and thus fine asymptotic properties.
\end{itemize}

To feel the generality of this protocol one can for instance consider mini-batch stochastic approximation strategies based on discretization of standard continuous time dynamical systems with known Lyapunov stability. Prominent examples include the heavy ball momentum method \cite{attouch2000heavy} and ADAM algorithm \cite{barakat2018convergence}, commonly proposed in deep learning libraries, as well as INDIAN introduced and studied in \cite{castera2019inertial}.

\bigskip

\noindent
{\bf Acknowledgements.} The authors acknowledge the support of AI Interdisciplinary Institute ANITI funding, through the French ``Investing for the Future -- PIA3'' program
under the Grant agreementi ANR-19-PI3A-0004, Air Force Office of Scientific Research, Air Force Material Command, USAF, under grant numbers FA9550-19-1-7026, FA9550-18-1-0226, and ANR MasDol. J. Bolte acknowledges the support of ANR Chess, grant ANR-17-EURE-0010 and ANR OMS.

The authors would like to thank Lionel Thibault,  Sylvain Sorin, an anonymous referee for their careful readings of an early version of this work; we also thank Gersende Fort for her very valuable comments and discussions on stochastic approximation.

\end{document}